\documentclass[11pt, reqno]{amsart}

\usepackage{amsmath,amsfonts,amssymb,mathrsfs}
\usepackage{mathrsfs,extpfeil}
\usepackage{indentfirst, latexsym, enumerate}
\usepackage[retainorgcmds]{IEEEtrantools}
\usepackage{graphicx,epsfig,subcaption,float}
\usepackage{colortbl,bm}
\usepackage{hyperref}
\usepackage{cite}
\usepackage{upgreek}
\hypersetup{
	hyperindex=true
}

\hypersetup{hidelinks}

\title[Fokker-Planck equations on discrete infinite graphs]{Fokker-Planck equations on discrete infinite graphs}

\author[J. A. Carrillo]{Jos\'e A. Carrillo}
\address[Jos\'e A. Carrillo]{\newline Mathematical Institute, \newline
University of Oxford, Oxford OX2 6GG, UK.}
\email{jose.carrillo@maths.ox.ac.uk}

\author[X. Wang]{Xinyu Wang$^{*}$}
\address[Xinyu Wang]{	\newline School of Mathematics \newline Harbin Institute of Technology, Harbin  150001, People's Republic of China}
\email{wangxinyumath@hit.edu.cn}

\newtheorem{theorem}{Theorem}[section]
\newtheorem{lemma}{Lemma}[section]
\newtheorem{corollary}{Corollary}[section]
\newtheorem{proposition}{Proposition}[section]

\newtheorem{remark}{Remark}[section]

\newcommand{\bbr}{\mathbb R}

\newcommand{\bbz}{\mathbb Z}

\newcommand{\bbm} {\mathbb M}
\newcommand{\bbn} {\mathbb N}

\newcommand{\bT}{\mathcal T}

\def\charf {\mbox{{\text 1}\kern-.24em {\text l}}}

\newcommand*\di{\mathop{}\!\mathrm{d}}

\allowdisplaybreaks

\begin{document}

\thanks{\textbf{Acknowledgment.}
 JAC was partially supported by the Advanced Grant Nonlocal-CPD (Nonlocal PDEs for Complex Particle Dynamics: Phase Transitions, Patterns and Synchronization) of the European Research Council Executive Agency (ERC) under the European Union Horizon 2020 research and innovation programme (grant agreement No. 883363). JAC acknowledges support by the EPSRC grant EP/V051121/1. JAC was also partially supported by the “Maria de Maeztu” Excellence Unit IMAG, reference CEX2020-001105-M, funded by MCIN/AEI/10.13039/501100011033/. XYW was supported by the Natural Science Foundation of China (grants 123B2003), the China Postdoctoral Science Foundation (grants 2025M774290), and Heilongjiang Province Postdoctoral Funding (grants LBH-Z24167). The authors would like to express special thanks to Xiaoping Xue and Antonio Esposito for their valuable comments and suggestions, which greatly improved the quality of the paper. $^{*}$Corresponding author.}

\begin{abstract}
We study the gradient flow structure and long-time behavior of Fokker–Planck equations (FPE) on infinite graphs, along with a Talagrand-type inequality in this setting. We begin by constructing an infinite-dimensional Hilbert manifold structure, extending the approach of [S. N. Chow, W. Huang, Y. Li, H. M. Zhou, Arch. Ration. Mech. Anal., 203, 969-1008 (2012)] through a novel classification method to establish injectivity of the map from quotient space to tangent space and employing functional analysis techniques to prove surjectivity. Using a combination of the relative energy method, approximation techniques, and continuity arguments, we establish the global existence and asymptotic convergence of solutions to the infinite-dimensional ODE system associated with the FPE. Specifically, we demonstrate that the FPE admits a gradient flow structure, with solutions converging exponentially to the unique Gibbs distribution. Furthermore, we prove a local Talagrand-type inequality and compare the Hilbert manifold metric induced by our framework with classical Wasserstein distances. 
\end{abstract}
\maketitle
%
%
%
%
%
%
\section{Introduction}
\setcounter{equation}{0}
As the stochastic differential equation has become one of the most powerful and widely used tools in various applied fields, including finance, physics, chemistry, and biology \cite{R1989, S1980}, a lot of research have been devoted to studying the properties of the classical Fokker–Planck equations (FPE), which governs the time evolution of the probability distribution for stochastic processes. It is well known that the groundbreaking work of Jordan, Kinderlehrer, and Otto \cite{J-K-O, O2001} established that the FPE can be interpreted as the gradient flow of the \textit{free energy functional} on a Riemannian manifold formed by the space of probability measures endowed with the 2-Wasserstein metric. \textit{The free energy functional}-usually defined on the space of probability measures as a linear combination of potential and entropy terms- has found broad applications across multiple disciplines \cite{S-W-1997, W-H-Y}. In this paper, we consider the FPE for a \textit{free energy functional} defined on an infinite graph, and the gradient flow structure, long-time behavior, and Talagrand-type inequality are investigated.

To set up the stage, we begin with the classical free energy functional defined in Euclidean space. 
\begin{align*}
	\mathcal{F}_0(\mu)=\int_{\bbr^d} \Phi \mu +\beta\mu\log\mu\di x.
\end{align*}
Its gradient flow established in \cite{J-K-O} satisfies the following nonlinear FPE
\begin{align*}
\partial_t \mu =\nabla\cdot (\mu\nabla\Phi)  +\beta\Delta \mu.
\end{align*}
This seminal discovery initiated numerous important developments connecting four fundamental concepts in continuous spaces: free energy functionals, FPE, generalized notions of Ricci curvature, and optimal transportation theory (see monographs \cite{A-G-S, Villani2009, Villani20003} for comprehensive treatments). The subsequent work of Lott-Sturm-Villani \cite{L-V-2009, Sturm1, Sturm2} established a rigorous theory of Ricci curvature in length spaces, uncovering deep connections between entropy and curvature. However, while these theoretical frameworks are well-developed for continuous settings, their discrete counterparts, particularly for finite and infinite graphs, remain poorly understood, presenting significant challenges and open questions in the field for a long time.

\

Almost a decade and a half ago, Chow, Huang, Li, and Zhou \cite{C-H-L-Z} considered FPE for a free energy functional defined on a finite graph. They showed that the solution of FPE converges to the Gibbs measure asymptotically, which mimics the entropy dissipation property in a discrete space. In particular, their analysis specifically revealed that the FPE constitutes the gradient flow of the free energy functional on a Riemannian manifold equipped with a logarithmic-mean-type metric structure that generalizes classical Wasserstein-type geometries. At the same time, Mielke \cite{M2011,M2013} and Maas 
\cite{M-2011} independently established similar Riemannian metrics in finite graphs assumed to be associated with an irreducible and reversible Markov kernel. Since then, the concept related to optimal transportation on finite discrete graphs has been extensively studied \cite{E-M-2012,C-L-Z,M-F-L-S,M-F-S-2020,E-M-W-2019,E-2014,F-M-P-2022}. In particular, Che, Huang, Li, and Tetali \cite{C-H-L-T} showed that the solution of FPE derived in \cite{C-H-L-Z} will converge to the Gibbs measure exponentially fast, and established the discrete analog of Talagrand-type inequality. Another related direction is the cosh-type generalized gradient structure, which provides an important alternative to the logarithmic-mean-type metric structure used in discrete optimal transport. A well-known drawback of the logarithmic-mean-type structure for nonlocal diffusion is its strong reliance on strict positivity of the density. This difficulty is one of the reasons why non-metric cosh-type generalized gradient-flow structures, first introduced in the context of Markov chains and large-deviation principles by Mielke, Peletier, and Renger \cite{M-P-R-2014}, have attracted increasing attention in recent years. The advantage of the cosh-type structure is that it is naturally connected with the large-deviation principle for jump processes and Markov chains, and it often provides a more flexible non-metric generalized gradient-flow formulation \cite{M-P-S-2021,P-S-2023}.

\

Compared to a finite graph, the gradient flow of the free energy functional on an infinite graph has only been touched on until five years ago \cite{E-F-S-S-2021}. Esposito, Francesco, Schlichting, and Slepcev introduced graph analogues of the continuum nonlocal-interaction equation and interpreted them as gradient flows concerning a graph Wasserstein distance. Then, a series of works has been developed to study the nonlocal interaction equation on infinite graphs \cite{H-P-S-2023,E-L-2024,E-H-S-2025,Warren-2024}. However, their technique relies on the continuity of the graphon function, and they can only deal with continuous spaces, not discrete ones. Moreover, they do not establish functional inequalities and the long-time behavior of the gradient flow.  Therefore, the following questions arise naturally:
\vspace{0.2cm}

\begin{itemize}
	\item (Q1) Does there exist a corresponding discrete gradient flow of the free energy functional on infinite graphs?
	\item (Q2) If yes, under what conditions on initial data can we establish the long-time behavior and corresponding functional inequalities?
\end{itemize}
\vspace{0.2cm}

The main purpose of this paper is to answer the above questions. For (Q1), we first introduce the neighborhood relation of vertices via an undirected graph $\mathcal{G}=(\mathcal{V},\mathcal{E})$, where $\mathcal{V}$ is a countably infinite vertex set and $\mathcal{E}\subset\mathcal{V}\times\mathcal{V}$ denotes the set of undirected edges. For any edge $e\in\mathcal{E}$ connecting vertices $i,j\in\mathcal{V}$, we write $i\sim j$. Let $\bbm=(m_i)_{i\in\mathcal{V}}$ be a network weight function on $\mathcal{V}$, \[\bbm=(m_i), \quad \sum\limits_{i\in\mathcal{V}} m_i=1,\quad m_i>0,\quad \forall~ i\in\bbn.\]
Then, we obtain the weighted infinite graphs $\mathcal{G}=(\mathcal{V},\mathcal{E}, \bbm)$. Moreover, we define the norms
\[
\|\rho\|_{l^{\infty}}:=\sup_{i\in\mathbb N} |\rho_i|,
\qquad
\|\rho\|_{l_{\bbm}^2}:=
\left(\sum_{i=1}^{\infty} m_i|\rho_i|^2\right)^{\frac12}.
\]

\

Now, we introduce relative free energy functional as follows:
\[
\mathcal{F}(\rho|\bbm):=\sum_{i=1}^{\infty} m_i\Phi_i\rho_i +\beta \sum_{i=1}^{\infty} m_i\rho_i \log \rho_i,
\]
where  $\bbm = \{m_1, m_2,\dots, m_N,\dots\}$ is the network weight, \( \rho = (\rho_1, \rho_2, \dots, \rho_N,\dots)\) is the density function, and \( \Phi = (\Phi_1, \Phi_2, \dots, \Phi_N,\dots) \in l^{\infty} \) is a given potential. This can be seen as a discrete version of classical relative free energy functional \begin{align*}
	\mathcal{F}_0(\mu| m)=\int_{\bbr^d} \Phi \rho +\beta\rho\log\rho\di m,
\end{align*}
where $\mu=\rho \di m$ and $\rho$ can be regarded as density of $\mu$ with respect to measure $m$. 

\

We adapt a similar method in \cite{C-H-L-Z} to establish the gradient flow of $\mathcal{F}(\rho|\bbm)$ on the infinite graph $\mathcal{G}=(\mathcal{V},\mathcal{E},\bbm)$ as follows.
\begin{align*}
	\begin{aligned}
		\frac{\di\rho_i}{\di t} = 
		& \sum_{ \Phi_j > \Phi_i} m_j\left[ (\Phi_j + \beta \log \rho_j) -(\Phi_i + \beta \log \rho_i) \right] \rho_j \\
		&	+  \sum_{ \Phi_j < \Phi_i} m_j\left[ (\Phi_j + \beta \log \rho_j) - (\Phi_i + \beta \log \rho_i) \right] \rho_i \\
		&	+ \sum_{\Phi_j = \Phi_i}m_j \beta(\rho_j - \rho_i), \qquad \text{for } ~i \in \mathbb{N}:=\left\{1,2,3,...\right\} .
	\end{aligned}
\end{align*}
 Our discrete infinite FPE can be seen as the natural extension of a finite case in \cite{C-H-L-Z}. Compared to an undirected finite graph setting, we consider the special infinite graph $\mathcal{G}=(\mathcal{V},\mathcal{E},\bbm)$, named sender network (weight function only depends on the sender), which has been used to study emergent behavior of infinite particle systems in \cite{H-L-S,w1,H-L-2024,H-L-Y,w5,H-W-X-2025,w8}. In particular, we show that the Gibbs distribution
\[
\rho_i^* = \frac{e^{\frac{-\Phi_i}{\beta}}}{\sum_i m_ie^{\frac{-\Phi_i}{\beta}}}, \quad i\in\bbn, 
\]
is the unique stationary distribution of the FPE when $\beta>0$. Furthermore, \( \mathcal{F}(\rho|\bbm) \) attains its global minimum at \( \rho^* \). 

The first novelty of our results is to introduce a new classification method to ensure injectivity and apply functional analysis arguments to establish surjectivity. This allows for a fully rigorous infinite-dimensional Hilbert manifold structure for infinite graphs by adapting the approach of \cite{C-H-L-Z}. 

\

Regarding question (Q2), we first introduce the manifold 
\begin{equation}\label{New1}
	\mathcal{M}  = \left\{ \rho=\left(\rho_i\right)_{i=1}^{\infty} \in l_{\bbm}^2(\bbr) ~\Bigg|~ \sum_{i=1}^{\infty} m_i\rho_i = 1, ~\ \rho_i>0, ~\ i\in \bbn \right\}
\end{equation}
and 
\begin{equation}\label{New2}
	{\rm Int} ( \mathcal{M} )= \left\{ \rho=\left(\rho_i\right)_{i=1}^{\infty} \in l_{\bbm}^2(\bbr)~\Bigg|~ \sum_{i=1}^{\infty} m_i\rho_i = 1,~\inf_i\rho_i>0\right\}.
\end{equation}
 Then, we can show that the solution of FPE will converge to the Gibbs distribution exponentially fast under $l_{\bbm}^2$-metric
\[\|\rho(t)-\rho^{*}\|_{l_{\rm \bbm}^2}^2\le \frac{\sup_i\rho_i^{*}}{\inf_i\rho_i^{*}} \|\rho_0-\rho^{*}\|_{l_{\rm \bbm}^2}^2e^{-Ct} \]
as long as the initial value $\rho_0\in {\rm Int} ( \mathcal{M} )\cap l^{\infty}(\bbr)$. For more details, we refer to Theorem \ref{T2.1}.

To establish the exponential equilibration stated in Theorem~\ref{T2.1}, we combine the relative energy method with approximation and continuity arguments. 
The key innovation lies in analyzing the evolution of the supremum and infimum of the densities on vertices, which may not be attained by any individual vertex, thereby overcoming a long-standing difficulty in the study of infinite graphs. This is achieved by deriving upper and lower uniform-in-time bounds of the solutions.
Building upon these a priori estimates, Theorem~\ref{T4.1} further establishes the following local Talagrand-type inequality:
\begin{align*}
	d_{\Phi}^2(\mu,\nu) \le \kappa\, H(\nu |\mu),
\end{align*}
for all \( \mu, \nu \in {\rm Int}\mathcal{M} \cap l^{\infty}(\mathbb{R}) \).  
Here, $\kappa>0$ is a constant depending on $\mu$ and $\nu$, $H(\nu|\mu)$ is the relative entropy between $\mu$ and $\nu$ defined as
$$
H(\nu|\mu)=\sum\limits_{i=1}^{\infty}m_i\nu_i\log\frac{\nu_i}{\mu_i}
$$
and the potential function $\Phi=\left(\Phi_i\right)_{i=1}^{\infty}$ is given in terms of $\mu$ by $\Phi_i=-\log \mu_i$. Notice that the constant $\kappa$ depends both on the upper and lower bounds of $\mu$ and $\nu$. The local Talagrand inequality leads to a better understanding of the geometric structure induced by our Hilbert manifold.

\

Our analysis reveals a rigorous gradient flow formulation of the FPE on discrete infinite graphs, underpinning the understanding of entropy-driven dynamics in infinite-dimensional discrete settings.  At the same time, several important questions remain open. First, our construction relies on the strict positivity and uniform \(l^\infty\)-boundedness of the density. This assumption is used to guarantee the non-degeneracy of the logarithmic-mean-type mobility and to obtain the comparison estimates needed for the metric tensor and the well-posedness theory. It would be desirable to replace the uniform \(l^\infty\)-bound by weaker discrete moment or integrability conditions, closer in spirit to the L\'evy-measure framework in \cite{E-2014}, where nonlocal Fokker--Planck-type equations can be studied on the whole unbounded space. Developing such a theory on discrete infinite graphs is an interesting and challenging direction for future research. Another natural direction is to extend the present analysis from sender networks to more general connected infinite graphs, for instance integer lattice networks as in \cite{B-C-G,B-2019}. Functional inequalities on infinite graphs also require further investigation, including log-Sobolev inequalities \cite{O-V-2000} and \L{}ojasiewicz inequalities \cite{L-X-2022}. Finally, the gradient-flow structure and long-time behavior of time-varying graph continuity equations have recently been studied in \cite{C-E-M-2025}. Whether the present results can be extended to time-varying infinite graphs is also an interesting topic for future research.

\

The rest of this paper is organized as follows. In Section \ref{sec:2}, we introduce the infinite-dimensional Hilbert manifold. In Section \ref{sec:3}, we derive the gradient flow of the relative free energy functional and establish its long-time behavior. In our last Section \ref{sec:4}, we prove the Talagrand-type inequalities and compare our Hilbert manifold metric with 1-Wasserstein distance. 

\vspace{0.2cm}

\section{Hilbert manifold structure}\label{sec:2}
\setcounter{equation}{0}
In this section, we first adapt the argument of \cite{C-H-L-Z} to establish a Hilbert manifold structure on infinite graphs. We then compare the sender network considered in this paper with more general infinite graphs. Throughout this section, we restrict ourselves to densities satisfying
\[
0<\rho_{\inf}:=\inf_{i\in\bbn}\rho_i
\le
\rho_{\sup}:=\sup_{i\in\bbn}\rho_i
<\infty.
\]
This restriction is used to guarantee that the logarithmic terms and the Hilbert manifold metric are well defined. Before moving on further, we provide the following notation as a preparation.\newline

\noindent\textbf{Notation}: We introduce the $l^{\infty}$- and $l_{\bbm}^2$- spaces as follows:
	\begin{align*}
	l^{\infty}(\bbr)&:= \left\{ \rho=\left(\rho_i\right)_{i=1}^{\infty}\in \bbr^{\bbn} ~\Bigg|~ \|\rho\|_{l^{\infty}}:=\sup\limits_{i\in\bbn} |\rho_i| <\infty \right\},\\
l_{\bbm}^2(\bbr)&:= \left\{ \rho=\left(\rho_i\right)_{i=1}^{\infty}\in \bbr^{\bbn} ~\Bigg|~ \|\rho\|_{l_{\bbm}^2}:=\left(\sum_{i=1}^{\infty} m_i|\rho_i|^2\right)^{\frac{1}{2}} <\infty \right\}.
\end{align*}
Throughout this paper, if \(K:X\to Y\) is a linear map between two vector spaces, we denote by
		\[
		R(K):=K(X)\subset Y
		\]
		the range of \(K\), and by
		\[
		N(K):=\{x\in X:Kx=0\}\subset X
		\]
		the kernel of \(K\).

\subsection{Tangent space and bijection map}\label{sec:2.1}
First, we introduce the manifold
\begin{equation*}
	\mathcal{M}
	=
	\left\{
	\rho=(\rho_i)_{i=1}^{\infty}\in l_{\bbm}^2(\mathbb R)
	\ \Bigg|\
	\sum_{i=1}^{\infty}m_i\rho_i=1,\quad
	\rho_i>0,\quad i\in\mathbb N
	\right\}.
\end{equation*}
Define the constraint functional
\[
\mathcal{G}:l_{\bbm}^2(\bbr)\to\mathbb R,
\qquad
\mathcal{G}(\rho)
=
\sum_{i=1}^{\infty}m_i\rho_i-1.
\]
Then, for each \(\rho\in\mathcal M\), the tangent space is given by
\begin{equation}\label{T_pM}
\mathcal{T}_\rho\mathcal{M}
=
\left\{
\sigma\in l_{\bbm}^2(\mathbb R) 
~ \Bigg|~
\frac{\di}{\di \varepsilon}
\mathcal{G}(\rho+\varepsilon\sigma)
\Big|_{\varepsilon=0}
=
\sum_{i=1}^{\infty}m_i\sigma_i=0
\right\}.
\end{equation}
Second, we define the diagonal element as 
\begin{equation*}
	\odot = \left\{ \rho=\left(\rho_i\right)_{i=1}^{\infty} \in l_{\bbm}^2~\mid~ \rho_i = c, ~c\in\bbr \right\},
\end{equation*}
and it is easy to see
\begin{equation*}
	\odot^{\perp}=  \left\{ \sigma=\left(\sigma_i\right)_{i=1}^{\infty} \in l_{\bbm}^2(\bbr) ~\Bigg|~ \sum_{i=1}^{\infty} m_i\sigma_i = 0 \right\}=\bT_{\rho}  \mathcal{M}.
\end{equation*}
Note that for any $x\in l_{\bbm}^2$, it can be decomposed into 
\[x=x^{	\odot}+x^{	\odot^{\perp}}, \quad x^{	\odot}\in	\odot,\quad x^{	\odot^{\perp}}\in 	\odot^{\perp}. \]
Moreover, we recall the equivalence relation ``$\sim$'' in $l_{\bbm}^2$.
	\begin{align*}
		[p]\sim[q] \quad \text{if and only if } p_1-q_1=p_2-q_2=...=p_i-q_i=... ,~~\forall ~i\in\mathbb{N}.
	\end{align*}
	Then, let $\mathcal{W}$ be the quotient space $l_{\bbm}^2/\sim ~\cong ~ l_{\bbm}^2/	\odot$. In other words, for $p\in l_{\bbm}^2$, we consider its equivalence class
	\[[p]=\left\{(p_1+c,p_2+c,...)\quad c\in\bbr\right\},\]
and we know $[p]=[p^{	\odot^{\perp}}].$ \newline

	Now, we want to define a bijection from $\mathcal{W}$ to $\bT_{\rho}\mathcal{M}$. We introduce the map $\tau_{\Phi(\rho)}([p])=(\sigma_i)_{i=1}^{\infty}$ ~($\mathcal{W}\to \bT_{\rho} \mathcal{M} $):
\begin{align}\label{bijection}
	\sigma_i=\sum_{j=1}^{\infty}m_j\tau_{ij}^{\Phi}(\rho)(p_i-p_j),\end{align}
where
\[\tau_{ij}^{\Phi}(\rho)=\begin{cases}
	\displaystyle \rho_i,\qquad \qquad \qquad\Phi_i>\Phi_j,\\
	\displaystyle 	\rho_j,\qquad\qquad \qquad  \Phi_i<\Phi_j,\\
	\displaystyle 	\frac{\rho_i-\rho_j}{\log \rho_i-\log \rho_j} ,\quad \Phi_i=\Phi_j.
\end{cases}\]
Next, we show that \eqref{bijection} is well-defined and a bijection.
Note that 
	\begin{align}\label{B-2-2}
\min\left\{\rho_i,~\rho_j\right\}	\le	\frac{\rho_i-\rho_j}{\log \rho_i-\log \rho_j}\le \max\left\{\rho_i,~ \rho_j\right\}.
	\end{align}
First, it is easy to verify 
\[\sum_{i=1}^{\infty}m_i\sigma_i=\sum_{i=1}^{\infty}m_i\sum_{j=1}^{\infty}m_j\tau_{ij}^{\Phi}(\rho)(p_i-p_j)=0.\]
Second, we prove that \(\tau_{\Phi(\rho)}\) is injective. Since the argument requires a distinction between the cases where the supremum of the potential representative is attained, not attained, or infinite, we give the details in Appendix~\ref{App-A}.

Next, we need to show that it is surjection, i.e., for any $\sigma\neq0\in\bT_{\rho} \mathcal{M} $, there exists some $[p]\neq[0]$ such that $\tau_{\Phi(\rho)}([p])=(\sigma_i)_{i=1}^{\infty}$.
We redefine the map $\tilde{\tau}_{\Phi(\rho)}(p)=(\tilde{\sigma}_i)_{i=1}^{\infty}$ ~($l_{\bbm}^2\to l_{\bbm}^2$):
\begin{align}\label{bijection1}
	\tilde{\sigma}_i=\sum_{j=1}^{\infty}m_j\tau_{ij}^{\Phi}(\rho)(p_i-p_j).
	\end{align}
Since we have $\sup\limits_{i,j}|\tau_{ij}^{\Phi}|\le \|\rho\|_{l^\infty}$, we know the operator $\tilde{\tau}_{\Phi(\rho)}(p)$ is bounded linear continuous and therefore closed. Furthermore, we use $\tau_{ij}^{\Phi}(\rho)=\tau_{ji}^{\Phi}(\rho)$ to see it is a seld-adjoint operator. Therefore, the kernel of dual operator $\tilde{\tau}_{\Phi(\rho)}(p)^{*}$ is also $\odot$, which combines the [Corollary 2.18, \cite{Bre}] yield
\begin{align*}
	\odot=N(\tilde{\tau}_{\Phi(\rho)}(p)^{*})=R(\tilde{\tau}_{\Phi(\rho)}(p))^{\perp}.
\end{align*}
This further implies $R(\tilde{\tau}_{\Phi(\rho)}(p))=N(\tilde{\tau}_{\Phi(\rho)}(p)^{*})^{\perp}=	\odot^{\perp}=\bT_{\rho}  \mathcal{M} .$
Then, we return to the map $\tau_{\Phi(\rho)}([p])=(\sigma_i)_{i=1}^{\infty}$ and we have  \[R(\tau_{\Phi(\rho)}([p]))=R(\tilde{\tau}_{\Phi(\rho)}(p))=\bT_{\rho}  \mathcal{M} .\]
\subsection{Hilbert manifold metric}\label{sec:2.2}
Now, we can define the inner product in $\bT_{\rho} \mathcal{M} $ as follows:
\begin{align}\label{B-2-5}
	\begin{aligned}
	g_{\rho}^{\Phi}(\sigma^1, \sigma^2)&=\left\langle p^1,\sigma^2 \right\rangle_{l_{\bbm}^2}=\sum\limits_{i=1}^{\infty}m_i p_i^1\sigma_i^2\\
	&=\frac{1}{2}\sum\limits_{i=1}^{\infty}\sum\limits_{j=1}^{\infty}m_im_j\tau_{ij}^{\Phi}(\rho)(p_i^1-p_j^1)(p_i^2-p_j^2),
		\end{aligned}
\end{align}
where $p^1$ is the identification of $\sigma^1$ through the inverse map of $\tau_{\Phi(\rho)}([p])$.
Then we have 
\begin{align}\label{B-2-6}
	\begin{aligned}
		g_{\rho}^{\Phi}(\sigma^1, \sigma^1)&=\sum\limits_{i=1}^{\infty}m_i p_i^1\sigma_i^1\\
		&=\frac{1}{2}\sum\limits_{i=1}^{\infty}\sum\limits_{j=1}^{\infty}m_im_j\tau_{ij}^{\Phi}(\rho)(p_i^1-p_j^1)(p_i^1-p_j^1).
	\end{aligned}
\end{align}
This is equivalent to the norm $\|~\cdot~\|_{l_{\bbm}^2}$ for $p^1\in	\odot^{\perp}$ and $\sigma^1$ if we have $\inf\limits_{i\in\bbn}\rho_i>0$ and $\sup\limits_{i\in\bbn}\rho_i<\infty$. 
Since for $p^1\in	\odot^{\perp}$ 
\begin{align}\label{B-2-7}
	\begin{aligned}
		\sum\limits_{i=1}^{\infty}\sum\limits_{j=1}^{\infty}m_im_j(p_i^1-p_j^1)(p_i^1-p_j^1)&=	\sum\limits_{i=1}^{\infty}m_i|p_i^1|^2+\sum\limits_{j=1}^{\infty}m_j|p_j^1|^2-2\left\langle\sum\limits_{i=1}^{\infty}m_ip_i^1, \sum\limits_{j=1}^{\infty}m_jp_j^1 \right\rangle\\
		&= 2\|p\|_{l_{\bbm}^2}^2,
	\end{aligned}
\end{align}
and then we can get 
\begin{align}\label{B-2-8}
	\begin{aligned}
		g_{\rho}^{\Phi}(\sigma^1, \sigma^1)&=\frac{1}{2}\sum\limits_{i=1}^{\infty}\sum\limits_{j=1}^{\infty}m_im_j\tau_{ij}^{\Phi}(\rho)(p_i^1-p_j^1)(p_i^1-p_j^1)\\
		&\ge\frac{1}{2} \left(\inf\limits_{i\in\bbn}\rho_i\right)\sum\limits_{i=1}^{\infty}\sum\limits_{j=1}^{\infty}m_im_j(p_i^1-p_j^1)(p_i^1-p_j^1)=\left(\inf\limits_{i\in\bbn}\rho_i\right)\|p^1\|_{l_{\bbm}^2}^2
	\end{aligned}
\end{align}
and 
\begin{align}\label{B-2-9}
	\begin{aligned}
		g_{\rho}^{\Phi}(\sigma^1, \sigma^1)&=\frac{1}{2}\sum\limits_{i=1}^{\infty}\sum\limits_{j=1}^{\infty}m_im_j\tau_{ij}^{\Phi}(\rho)(p_i^1-p_j^1)(p_i^1-p_j^1)\\
		&\le \frac{1}{2}\left(\sup\limits_{i\in\bbn}\rho_i\right)\sum\limits_{i=1}^{\infty}\sum\limits_{j=1}^{\infty}m_im_j(p_i^1-p_j^1)(p_i^1-p_j^1)=\left(\sup\limits_{i\in\bbn}\rho_i\right)\|p^1\|_{l_{\bbm}^2}^2.
	\end{aligned}
\end{align}
Furthermore, we obtain 
\begin{align}\label{B-2-10}
\|\sigma^1\|_{l_{\bbm}^2}^2&=\sum\limits_{i=1}^{\infty}m_i\Bigg|\sum\limits_{j=1}^{\infty}m_j\tau_{ij}^{\Phi}(\rho)(p_i^1-p_j^1)\Bigg|^2\nonumber\\
&\le\sum\limits_{i=1}^{\infty}m_i\sum\limits_{j=1}^{\infty}m_j(\tau_{ij}^{\Phi}(\rho))^2\Bigg|p_i^1-p_j^1\Bigg|^2\nonumber\\
&\le\left(\sup\limits_{i\in\bbn}\rho_i\right)\left(\sum\limits_{i=1}^{\infty}m_i\sum\limits_{j=1}^{\infty}m_j\tau_{ij}^{\Phi}(\rho)\Bigg|p_i^1-p_j^1\Bigg|^2\right) \nonumber\\
    &=2\left(\sup\limits_{i\in\bbn}\rho_i\right)	g_{\rho}^{\Phi}(\sigma^1, \sigma^1),
\end{align}
and 
\begin{align}\label{B-2-11}
\begin{aligned}
	\|\sigma^1\|_{l_{\bbm}^2}^2&=\sum\limits_{i=1}^{\infty}m_i\Bigg|\sum\limits_{j=1}^{\infty}m_j\tau_{ij}^{\Phi}(\rho)(p_i^1-p_j^1)\Bigg|^2\ge\left(\inf\limits_{i\in\bbn}\rho_i\right)^2	\|p^1\|_{l_{\bbm}^2}^2\ge\frac{\left(\inf\limits_{i\in\bbn}\rho_i\right)^2	}{\sup\limits_{i\in\bbn}\rho_i} g_{\rho}^{\Phi}(\sigma^1, \sigma^1).
    \end{aligned}
\end{align}
Here, we used that
\begin{align*}
\begin{aligned}
\Bigg(\sum\limits_{i=1}^{\infty}m_i\Bigg|\sum\limits_{j=1}^{\infty}m_j\tau_{ij}^{\Phi}(\rho)(p_i^1-p_j^1)\Bigg|^2\Bigg)^{\frac{1}{2}}\|p^1\|_{l_{\bbm}^2}&\ge\sum\limits_{i=1}^{\infty}m_i\Big\langle\sum\limits_{j=1}^{\infty}m_j\tau_{ij}^{\Phi}(\rho)(p_i^1-p_j^1), p_i^1 \Big\rangle\\
&=\frac{1}{2}\sum\limits_{i=1}^{\infty}m_i\sum\limits_{j=1}^{\infty}m_j\Big\langle\tau_{ij}^{\Phi}(\rho)(p_i^1-p_j^1), p_i^1-p_j^1 \Big\rangle\\
&\ge\left(\inf\limits_{i\in\bbn}\rho_i\right)	 \frac{1}{2}\sum\limits_{i=1}^{\infty}m_i\sum\limits_{j=1}^{\infty}m_j\Big\langle p_i^1-p_j^1, p_i^1-p_j^1 \Big\rangle\\
&\ge \left(\inf\limits_{i\in\bbn}\rho_i\right)	\|p^1\|_{l_{\bbm}^2}^2.
    \end{aligned}
\end{align*}

Let \(g_\rho^\Phi\) be the family of inner products on the tangent spaces
\(\mathcal T_\rho\mathcal M\) constructed above. Formally, for
\(\rho_1,\rho_2\in\mathcal M\), one may define the Hilbert manifold metric by
\begin{equation}\label{B-2-12}
	d_\Phi(\rho_1,\rho_2)
	:=
	\inf_{\gamma}\mathcal L(\gamma),
\end{equation}
where the infimum is taken over all piecewise \(C^1\) curves
\[
\gamma:[0,1]\to\mathcal M
\]
such that
\[
\gamma(0)=\rho_1,\qquad \gamma(1)=\rho_2.
\]
The arc length of such a curve is defined by
\begin{equation}\label{B-2-13}
	\mathcal L(\gamma)
	:=
	\int_0^1
	\sqrt{
	g_{\gamma(t)}^\Phi(\dot\gamma(t),\dot\gamma(t))
	}
	\,\di t.
\end{equation}

However, on the whole space \(\mathcal M\), this Hilbert manifold metric may degenerate, since the logarithmic-mean-type mobility may lose uniform
non-degeneracy near the boundary of \(\mathcal M\). Therefore, the metric
properties are established below on each fixed uniformly positive and bounded
layer. For fixed constants \(0<a<b<\infty\), define
\[
\mathcal C_{a,b}
:=
\left\{
\rho=(\rho_i)\in\mathcal M
~\Big|~
a\le \rho_i\le b,\quad i\in\bbn
\right\}.
\]
For \(\rho_1,\rho_2\in\mathcal C_{a,b}\), we define the restricted Hilbert manifold metric
\[
	d_\Phi^{a,b}(\rho_1,\rho_2)
	:=
	\inf_{\gamma\subset\mathcal C_{a,b}}\mathcal L(\gamma),
\]
where the infimum is taken over all piecewise \(C^1\) curves
\(\gamma:[0,1]\to\mathcal C_{a,b}\) satisfying
\[
\gamma(0)=\rho_1,\qquad \gamma(1)=\rho_2.
\]
When the layer \(\mathcal C_{a,b}\) is fixed and clear from the context, we simply
write \(d_\Phi\) instead of \(d_\Phi^{a,b}\).
Notice that \(\mathcal C_{a,b}\) is convex. Indeed, if
\(\rho_1,\rho_2\in\mathcal C_{a,b}\), then the linear interpolation
\[
\gamma(t)=(1-t)\rho_1+t\rho_2,\qquad t\in[0,1],
\]
also belongs to \(\mathcal C_{a,b}\). Hence any two points in
\(\mathcal C_{a,b}\) can be connected by an admissible piecewise \(C^1\) curve.
We now verify that \(d_\Phi\) satisfies the axioms of a metric on
\(\mathcal C_{a,b}\). Non-negativity follows directly from the definition of
\(\mathcal L\), since \(g_\rho^\Phi\) is positive definite. Moreover, for every
\(\rho\in\mathcal C_{a,b}\), the constant curve \(\gamma(t)\equiv\rho\) is admissible
and satisfies \(\mathcal L(\gamma)=0\). Hence
\[
d_\Phi(\rho,\rho)=0.
\]

To prove symmetry, let \(\gamma:[0,1]\to\mathcal C_{a,b}\) be an admissible
piecewise \(C^1\) curve joining \(\rho_1\) to \(\rho_2\). Define the reversed curve
\[
\widetilde\gamma(t):=\gamma(1-t).
\]
Then \(\widetilde\gamma\) joins \(\rho_2\) to \(\rho_1\), remains in
\(\mathcal C_{a,b}\), and
\[
\dot{\widetilde\gamma}(t)=-\dot\gamma(1-t).
\]
Since \(g_{\gamma(t)}^\Phi\) is quadratic in the tangent vector, we have
\[
\mathcal L(\widetilde\gamma)=\mathcal L(\gamma).
\]
Taking the infimum over all admissible curves gives
\[
d_\Phi(\rho_1,\rho_2)=d_\Phi(\rho_2,\rho_1).
\]

The triangle inequality follows from concatenation of curves. Let
\(\rho_1,\rho_2,\rho_3\in\mathcal C_{a,b}\). For any \(\varepsilon>0\), choose
admissible piecewise \(C^1\) curves \(\gamma_1\) from \(\rho_1\) to \(\rho_2\) and
\(\gamma_2\) from \(\rho_2\) to \(\rho_3\) such that
\[
\mathcal L(\gamma_1)\le d_\Phi(\rho_1,\rho_2)+\varepsilon,
\qquad
\mathcal L(\gamma_2)\le d_\Phi(\rho_2,\rho_3)+\varepsilon.
\]
Their concatenation is again a piecewise \(C^1\) curve in \(\mathcal C_{a,b}\) joining
\(\rho_1\) to \(\rho_3\), and its length is
\[
\mathcal L(\gamma_1)+\mathcal L(\gamma_2).
\]
Therefore,
\[
d_\Phi(\rho_1,\rho_3)
\le
d_\Phi(\rho_1,\rho_2)
+
d_\Phi(\rho_2,\rho_3)
+
2\varepsilon.
\]
Letting \(\varepsilon\to0\), we obtain
\[
d_\Phi(\rho_1,\rho_3)
\le
d_\Phi(\rho_1,\rho_2)
+
d_\Phi(\rho_2,\rho_3).
\]

It remains to prove non-degeneracy. Since every admissible curve stays in
\(\mathcal C_{a,b}\), we have
\[
a\le\inf_{i\in\bbn}\gamma_i(t)
\le
\sup_{i\in\bbn}\gamma_i(t)\le b,
\qquad t\in[0,1].
\]
In particular, the logarithmic-mean-type mobility is uniformly non-degenerate and
bounded along the curve. By \eqref{B-2-10}, for every tangent vector
\(\sigma\in\mathcal T_{\gamma(t)}\mathcal M\),
\[
\|\sigma\|_{l_{\bbm}^2}^2
\le
2b\,g_{\gamma(t)}^\Phi(\sigma,\sigma).
\]
Applying this estimate to \(\sigma=\dot\gamma(t)\), we obtain
\[
g_{\gamma(t)}^\Phi(\dot\gamma(t),\dot\gamma(t))
\ge
\frac{1}{2b}
\|\dot\gamma(t)\|_{l_{\bbm}^2}^2.
\]
Hence
\[
\mathcal L(\gamma)
\ge
\frac{1}{\sqrt{2b}}
\int_0^1
\|\dot\gamma(t)\|_{l_{\bbm}^2}
\,\di t.
\]
By the triangle inequality in \(l_{\bbm}^2\),
\[
\int_0^1
\|\dot\gamma(t)\|_{l_{\bbm}^2}
\,\di t
\ge
\left\|
\int_0^1\dot\gamma(t)\,\di t
\right\|_{l_{\bbm}^2}
=
\|\rho_2-\rho_1\|_{l_{\bbm}^2}.
\]
Therefore,
\[
\mathcal L(\gamma)
\ge
\frac{1}{\sqrt{2b}}
\|\rho_2-\rho_1\|_{l_{\bbm}^2}.
\]
Taking the infimum over all admissible curves gives
\[
d_\Phi(\rho_1,\rho_2)
\ge
\frac{1}{\sqrt{2b}}
\|\rho_2-\rho_1\|_{l_{\bbm}^2}.
\]
Consequently, if \(d_\Phi(\rho_1,\rho_2)=0\), then
\[
\|\rho_2-\rho_1\|_{l_{\bbm}^2}=0,
\]
and hence \(\rho_1=\rho_2\). Thus \(d_\Phi\) is non-degenerate.

Therefore, for every fixed \(0<a<b<\infty\), \((\mathcal C_{a,b},d_\Phi)\) is a metric space.
Moreover, on \(\mathcal C_{a,b}\), the metric tensor \(g_\rho^\Phi\) is well defined and
uniformly comparable with the ambient \(l_{\bbm}^2\)-topology. Indeed, for every
\(\rho\in\mathcal C_{a,b}\) and \(\sigma\in\mathcal T_\rho\mathcal M\), estimate
\eqref{B-2-10} gives
\[
\|\sigma\|_{l_{\bbm}^2}^2
\le
2b\,g_\rho^\Phi(\sigma,\sigma),
\]
while the corresponding upper bound follows from \eqref{B-2-11}. Hence \(d_\Phi\)
is the Hilbert manifold metric induced by \(g_\rho^\Phi\) on the fixed layer
\(\mathcal C_{a,b}\).

\subsection{A shift from infinite graphs to sender networks}\label{sec:2.3}
In this subsection, we recall some infinite graphs and compare them with the sender network. 
\subsubsection{Generally fully connected infinite graph}\label{sec:2.3.1}

In emergent dynamics, they also consider the following generally fully connected infinite graphs $\mathcal{G}=(\mathcal{V},\mathcal{E}, A=(a_{ij}))$ as follows (see \cite{H-L-2024, H-L-S, H-L-Y}). 
For a given infinite network matrix ${A} = (A_{ij})$, we set 
\[  \| {A} \|_{-\infty, 1} := \inf_{i \in \mathbb{N}} \sum_{j=1}^{\infty}a_{ij} \quad \mbox{and} \quad   \|{A} \|_{\infty, 1} :=  \sup_{i \in \mathbb{N}} \sum_{j=1}^{\infty}a_{ij}.\]
We say that infinite graphs $\mathcal{G}=(\mathcal{V},\mathcal{E}, A=(a_{ij}))$ is generally fully connected means
	\[ \| A\|_{-\infty,1} > 0, \quad \exists ~\tilde{\mathbf{A}} = \{\tilde{a}_j\}_{j\in\mathbb{N}} \in l^1 \quad \left(\sum\limits_{j=1}^{\infty}\tilde{a}_j<\infty\right) \quad \mbox{such that} \quad \frac{a_{ij}}{\sum_{k=1}^\infty a_{ik}} \ge \tilde{a}_j > 0.\]
We now clarify the relation between this notion and the sender network used in the 
present paper. In our setting, the interaction weight is determined by the sequence
\(\bbm=(m_j)_{j\in\bbn}\), and the corresponding network matrix is
\[
a_{ij}=m_j,\qquad i,j\in\bbn.
\]
Thus each row of \(A\) coincides with the weight sequence \(\bbm\). Since
$
\sum_{j=1}^{\infty}m_j=1,
$
we have
\[
\sum_{j=1}^{\infty}a_{ij}=\sum_{j=1}^{\infty}m_j=1,\qquad i\in\bbn,
\]
and hence
\[
\|A\|_{-\infty,1}
=
\|A\|_{\infty,1}
=
1.
\]
Moreover,
\[
\frac{a_{ij}}{\sum_{k=1}^{\infty}a_{ik}}
=
m_j.
\]
Therefore, by taking
		$
		\tilde a_j=m_j,~j\in\bbn,
		$
		we obtain
		\[
		\frac{a_{ij}}{\sum_{k=1}^{\infty}a_{ik}}
		=
		\tilde a_j>0,
		\qquad i,j\in\bbn,
		\]
		and
		\[
	\sum_{j=1}^{\infty}\tilde a_j
		=
		\sum_{j=1}^{\infty}m_j
		=
		1<\infty.
		\]
Hence the sender network considered in this paper is a particular example of a generally fully connected infinite graph.
\subsubsection{Locally finite infinite graph}\label{sec:2.3.2}
 Another important example of infinite graphs is locally finite infinite graphs. Locally finite means each vertex can only be affected by a finite neighborhood, i.e.,
\[\sup\limits_{i}\#\left\{j\in\mathcal{N}(i)~\Bigg|~a_{ij}>0\right\}<\infty.\]
Like infinite lattice graph $\bbz^d$ and other examples presented in \cite{B-C-G,B-2019}. If we have
\[\bar{m}_i=\sum_{j\in\mathcal{N}(i)}a_{ij} \quad \text{and }\quad m_{\rm sum}:=\sum\limits_{i=1}^{\infty}\bar{m}_i<\infty.\]
Then, we can let \[\frac{\bar{m}_i}{m_{\rm sum}}=:m_i\] be the new weight from node $i$ to each node $k\in\mathcal{V}$. This gives us a new sender network from a locally finite infinite graph. A similar procedure can be done for infinite graphs that weight $A=(a_{ij})$ satisfies 
\[\sum\limits_{i=1}^{\infty}\sum\limits_{j=1}^{\infty}a_{ij}<\infty.\]
Also note that we do not require the graph to be symmetric, only that the weights of the graph be summable. The summability of the weights indicates that the network strength has enough decay at infinity so that they can become the bottom space weights of the solution in the following sections.

\section{Gradient flow of free energy functional on infinite graph}\label{sec:3}
\setcounter{equation}{0}
In this section, we establish the gradient flow of the free energy functional and show its long-time behavior.

Let \( \beta \geq 0 \) be a fixed constant. We recall the relative free energy functional \( \mathcal{F} \) on the manifold \( \mathcal{M} \) as:
\begin{equation*}
	\mathcal{F}(
\rho| \bbm) := \sum_{i=1}^{\infty} m_i\Phi_i \rho_i + \beta \sum_{i=1}^{\infty}m_i \rho_i \log \rho_i, \quad \sum\limits_{i=1}^{\infty}m_i=1, \quad m_i>0, \quad\forall~ i \in \bbn,
\end{equation*}
where \( \rho = (\rho_1, \rho_2, \dots, \rho_N,\dots) \in  \mathcal{M}  \), and \( \Phi = (\Phi_1, \dots, \Phi_N,\dots) \in l^{\infty} \) is a given potential. Then the gradient flow of \( \mathcal{F} \) with respect to the Hilbert manifold metric \( g^\Phi \) on \( \mathcal{M} \) is given by:
\begin{equation*}
	\frac{\di \rho}{\di t} = -\operatorname{grad} \mathcal{F}(\rho|\bbm),
\end{equation*}
where \( \operatorname{grad} \mathcal{F}(
\rho| \bbm) \) is the Hilbert manifold gradient of \( \mathcal{F} \) at point \( \rho \) under the metric \( g_{\rho}^\Phi \).
Let \( \mathrm{diff}\,\mathcal{F} \) denote the differential (Fréchet derivative) of \( \mathcal{F} \), which lies in the cotangent space. Then the gradient flow condition is equivalent to:
\begin{equation}\label{C-1-1}
	g_\rho^\Phi\left( \frac{\di\rho}{\di t}, \sigma \right) = - \langle \mathrm{diff} \,\mathcal{F}(
\rho| \bbm) ,\sigma\rangle_{l_{\bbm}^2}, \quad \forall ~\sigma \in \bT_\rho  \mathcal{M} .
\end{equation}
It is easy to verify that:
\begin{equation*}
	\mathrm{diff}\,\mathcal{F}(
\rho| \bbm) = \left( \Phi_i + \beta(1 + \log \rho_i) \right)_{i=1}^{\infty}.
\end{equation*}
Using this and the identification in Section \ref{sec:2}, one can obtain an explicit expression for the gradient flow on infinite graphs.

\begin{theorem}\label{T2.1}
	Let \( \mathcal{G} = (\mathcal{V}, \mathcal{E},\bbm) \) be an infinite graph with vertex set \( \mathcal{V} = \{a_1, a_2, \dots, a_N, \dots\} \), edge set \( \mathcal{E} \), weight function $\bbm$, a potential \( \Phi = (\Phi_i)_{i=1}^{\infty}\in l^{\infty}(\bbr) \) defined on \( \mathcal{V} \), and a constant \( \beta \geq 0 \). Then, the following assertions hold:

\begin{enumerate}
	\item The gradient flow of \( \mathcal{F} \) on the Hilbert manifold \( ( \mathcal{M} , d_\Phi) \) of probability densities on \( \mathcal{V} \) is given by the discrete infinite FPE:
	\begin{align}\label{FP}
		\begin{aligned}
		\frac{\di\rho_i}{\di t} = 
		& \sum_{ \Phi_j > \Phi_i} m_j\left[ (\Phi_j + \beta \log \rho_j) -(\Phi_i + \beta \log \rho_i) \right] \rho_j \\
	&	+  \sum_{ \Phi_j < \Phi_i} m_j\left[ (\Phi_j + \beta \log \rho_j) - (\Phi_i + \beta \log \rho_i) \right] \rho_i \\
	 &	+ \sum_{\Phi_j = \Phi_i}m_j \beta(\rho_j - \rho_i), \qquad \text{for } ~i \in \mathbb{N} .
	 \end{aligned}
	\end{align}
	
	\item For all \( \beta > 0 \), the Gibbs distribution
	\[
	\rho_i^* = \frac{e^{\frac{-\Phi_i}{\beta}}}{\sum_i m_ie^{\frac{-\Phi_i}{\beta}}}, \quad i\in\bbn, 
	\]
	is the unique stationary distribution of the FPE. Furthermore, \( \mathcal{F} \) attains its global minimum at \( \rho^* \).
	
	\item For any \( \beta > 0 \), there exists a unique solution
	\[
	\rho(t): [0, \infty) \to {\rm Int} ( \mathcal{M} )\cap l^{\infty}(\bbr)
	\]
	to the FPE with initial value \( \rho(0) = \rho_0 \in {\rm Int} ( \mathcal{M} ) \cap l^{\infty}(\bbr)\), such that:
	
	\begin{itemize}
		\item[(a)] \( \mathcal{F}(
\rho|\bbm) \) is non-increasing in time \( t \);

\item[(b)]  there exists some constant $C_1$ and $C_2$ depend on initial value such that \[0<C_1\le\inf\limits_{0\le t<\infty}\inf\limits_{i\in\bbn} \rho_i(t)\le \sup\limits_{0\le t<\infty}\sup\limits_{i\in\bbn} \rho_i(t)\le C_2<\infty;\]

		\item[(c)] \( \rho(t)\) converges to \(\rho^* \) exponentially under the $l_{\rm \bbm}^2$-metric as \( t \to +\infty \).
	\end{itemize}
\end{enumerate}
\end{theorem}

\begin{remark}
We now briefly explain the correspondence between the above discrete formulation and the classical continuous FPE. In the continuous setting, if \(\mu=\rho\,\di m\), the relative free energy is given by
	\[
	\mathcal F_0(\mu|m)
	=
	\int_{\mathbb R^d}\Phi(x)\rho(x)\di m(x)
	+
	\beta\int_{\mathbb R^d}\rho(x)\log\rho(x)\di m(x).
	\]
	In our discrete infinite-graph setting, the reference measure \(m\) is replaced by the atomic measure
	\[
	m=\sum_{i=1}^{\infty}m_i\delta_i,
	\]
	and the density \(\rho\) is represented by the sequence \(\rho=(\rho_i)_{i\in\bbn}\). 
	The FPE above should also be understood as a discrete counterpart of the Wasserstein gradient flow of the relative free energy. In the continuous case, the FPE can be written formally as
	\[
	\partial_t\rho
	=
	\nabla\cdot\left(\rho\nabla
	\left(\Phi+\beta\log\rho\right)\right),
	\]
	where \(\Phi+\beta\log\rho\) is the first variation of the relative free energy. In the graph setting, spatial gradients are replaced by differences along pairs of vertices,
	\[
	\left(\Phi_j+\beta\log\rho_j\right)
	-
	\left(\Phi_i+\beta\log\rho_i\right),
	\]
	and the transport mobility is chosen according to the upwind/logarithmic-mean-type structure introduced in the finite-graph framework of \cite{C-H-L-Z,M-2011,M2011}. 
In this sense, our formulation should be understood as a structural discrete analogue of the continuous FPE, rather than as a direct discretization or particle-in-cell approximation.  
\end{remark}
	\begin{remark}
		We give a concrete example illustrating the assumptions and the asymptotic convergence result. 
		Let
		\[
		m_i=2^{-i},\qquad \Phi_i=\frac{1}{i+1},\qquad i\in\bbn.
		\]
		Then \(m_i>0\), \(\sum_{i=1}^{\infty}m_i=1\), and 
		\(\Phi\in l^\infty(\bbr)\) with \(0<\Phi_i\le 1/2\). The corresponding Gibbs equilibrium is
		\[
		\rho_i^\ast
		=
		\frac{e^{-\Phi_i/\beta}}{Z_\beta},
		\qquad
		Z_\beta
		:=
		\sum_{k=1}^{\infty}m_k e^{-\Phi_k/\beta}.
		\]
		Moreover, since \(0<\Phi_i\le 1/2\), we have
		\[
		e^{-1/(2\beta)}\le Z_\beta\le 1,
		\qquad
		e^{-1/(2\beta)}\le \rho_i^\ast\le e^{1/(2\beta)},
		\quad i\in\bbn.
		\]
		Hence \(\rho^\ast\) is strictly positive and uniformly bounded. As an admissible initial value $\rho_0=(\rho_{i0})$, one may take
		\[
		\rho_{i0}
		=
		\frac{1+\varepsilon_i}
		{\sum_{k=1}^{\infty}m_k(1+\varepsilon_k)},
		\]
		where \((\varepsilon_i)_{i\in\bbn}\in l^\infty(\bbr)\) satisfies
		\[
		-1+\delta\le \varepsilon_i\le C
		\]
		for some constants \(\delta\in(0,1)\) and \(C>0\). Then \(\rho_0\) is a strictly positive and uniformly bounded probability density. Therefore, the global well-posedness and long-time convergence results apply, and the corresponding solution of the discrete FPE converges exponentially to \(\rho^\ast\) in the $l_{\bbm}^2$-metric stated in Theorem~\ref{T2.1}.
	\end{remark}
\begin{remark}\label{R3.1}
	Compared to finite graph settings \cite{C-H-L-T, C-H-L-Z, C-L-Z}, our proof needs to be more careful not only to deal with the lower bound of solutions, but also the upper bound. Moreover, the infimum and supremum may not be reached by any single vertex in an infinite graph, so we need to use a novel approximate method to estimate the derivatives of the infimum and supremum of solutions.
\end{remark}
\begin{proof}
	
	\noindent $\divideontimes$ (1): We use the symmetry of $\tau_{ij}^{\Phi}(\rho)$ to see
	\begin{align*}
			g_\rho^\Phi\left( \frac{\di\rho}{\di t}, \sigma \right)&=\sum\limits_{i\in\bbn}m_i\frac{\di \rho_i}{\di t} p_i=\sum\limits_{i\in\bbn}m_i\left(-\Phi_i - \beta(1 + \log \rho_i) \right)\sigma_i\\
			&=\sum\limits_{i\in\bbn}m_i\left(-\Phi_i - \beta( \log \rho_i) \right)\sigma_i\\
			&=\sum\limits_{i\in\bbn}m_i\left(\left(-\Phi_i - \beta( \log \rho_i)\right) \sum_{j\in\bbn}m_j\tau_{ij}^{\Phi}(\rho)(p_i-p_j)\right)\\
				&=\frac{1}{2}\sum\limits_{i\in\bbn}\sum_{j\in\bbn}m_im_j\left(\left(\Phi_j + \beta( \log \rho_j)-\Phi_i - \beta( \log \rho_i)\right) \tau_{ij}^{\Phi}(\rho)(p_i-p_j)\right)\\
				&=\sum\limits_{i\in\bbn}m_i\Big[\sum_{j\in\bbn}m_j\left(\Phi_j + \beta( \log \rho_j)-\Phi_i - \beta(\log \rho_i)\right) \tau_{ij}^{\Phi}(\rho)\Big]p_i.
	\end{align*}
	This implies 
	\begin{align*}
		\frac{\di \rho_i}{\di t}&=\sum_{j\in\bbn}m_j\left(\Phi_j + \beta(\log \rho_j)-\Phi_i - \beta( \log \rho_i)\right) \tau_{ij}^{\Phi}(\rho)\\
		& =\sum_{ \Phi_j > \Phi_i} m_j\left[ (\Phi_j + \beta \log \rho_j) -(\Phi_i + \beta \log \rho_i) \right] \rho_j \\
		&\quad+ \sum_{ \Phi_j < \Phi_i} m_j\left[ (\Phi_j + \beta \log \rho_j) - (\Phi_i + \beta \log \rho_i) \right] \rho_i \\
		&\quad+ \sum_{\Phi_j = \Phi_i}m_j \beta(\rho_j - \rho_i), \quad \text{for } ~i = 1,2, \dots .
	\end{align*}
		\noindent $\divideontimes$ (2): We note that the stationary point satisfies
		\begin{align*}
			g_\rho^\Phi\left( \frac{\di\rho^{*}}{\di t}, \sigma \right)
			=
			\sum_{i\in\bbn}m_i\frac{\di \rho_i^{*}}{\di t} p_i
			=
			\sum_{i\in\bbn}m_i\left(-\Phi_i-\beta\log \rho_i^{*}\right)\sigma_i
			=
			0
		\end{align*}
		for every \(\sigma\in\bT_{\rho}\mathcal{M}\). Since
		\[
		\bT_{\rho}\mathcal{M}
		=
		\left\{
		\sigma\in l_{\bbm}^2(\bbr)
		\ \Bigg|\
		\sum_{i=1}^{\infty}m_i\sigma_i=0
		\right\},
		\]
		we may, for each fixed \(i\ge2\), choose
		\[
		\sigma_i=\frac{1}{m_i},\qquad
		\sigma_1=-\frac{1}{m_1},\qquad
		\sigma_k=0\quad \text{for }k\neq 1,i.
		\]
		This choice satisfies \(\sigma\in\bT_\rho\mathcal M\). Substituting it into the
		stationarity condition gives
		\[
		-\Phi_i-\beta\log\rho_i^\ast
		+
		\Phi_1+\beta\log\rho_1^\ast
		=
		0.
		\]
		Therefore,
		\[
		\Phi_i+\beta\log\rho_i^\ast
		=
		\Phi_1+\beta\log\rho_1^\ast,
		\qquad i\ge2.
		\]
		Equivalently,
		\[
		\rho_i^\ast
		=
		e^{(\Phi_1-\Phi_i)/\beta}\rho_1^\ast.
		\]
		Using the normalization condition \(\sum_i m_i\rho_i^\ast=1\), we obtain
		\[
		\rho_1^\ast
		=
		\frac{1}
		{\sum_{i=1}^{\infty}m_i e^{(\Phi_1-\Phi_i)/\beta}},
		\]
		and hence
		\[
		\rho_i^\ast
		=
		\frac{e^{(\Phi_1-\Phi_i)/\beta}}
		{\sum_{k=1}^{\infty}m_k e^{(\Phi_1-\Phi_k)/\beta}}
		=
		\frac{e^{-\Phi_i/\beta}}
		{\sum_{k=1}^{\infty}m_k e^{-\Phi_k/\beta}},
		\qquad i\in\bbn.
		\]
		\vspace{0.2cm}
		
			\noindent $\divideontimes$ (3): Before moving to the proof details, we summarize our proof strategy below.
			\begin{itemize}
				\item  First, we define relative energy
				\begin{align}\label{C-1-3}
					L(t)=\sum\limits_{i=1}^{\infty}m_i\frac{(\rho_i(t)-\rho_i^{*})^{2}}{\rho_i^{*}},	\quad \text{and show that }	\quad\frac{\di L(t)}{\di t}\le-\frac{\beta}{2}\frac{\left(\inf_i\rho_i(t)\right)}{\left(\sup_i \rho_i(t)\right)}\frac{\left(\inf_i \rho_i^{*}\right)}{\left(\sup_i \rho_i^{*}\right)}L(t).
				\end{align}
                In the above process, we require 
               \[ 0<\inf\limits_{i\in\bbn} \rho_i(t)\le \sup\limits_{i\in\bbn} \rho_i(t)<\infty.\]
					\item	Second, we show that there exist some $j_0(t)\le N_{0}$ ($N_{0}$ is independent of time) such that \[\rho_{j_0(t)}(t)>1-\delta>0 \quad (0<\delta<1).\]
						\item   Third, we prove that the solutions will stay in the invariant set $\mathcal{C}$ in $l_{\bbm}^2(\bbr)$ forever, where 
						\begin{align}\label{C-2-8}
							\mathcal{C}:=\left\{\rho\in \mathcal{M}~~\Big|~~0<C_1\le\inf\limits_{i\in\bbn} \rho_i\le \sup\limits_{i\in\bbn} \rho_i\le C_2<\infty\right\}.
						\end{align}
                        More precisely, we show that \[ \frac{\di}{\di t}\inf\limits_{i\in\bbn} \rho_i(t)>0\quad \text{when} \quad\inf\limits_{i\in\bbn} \rho_i(t)=C_1,\]\[~ \frac{\di}{\di t}\sup\limits_{i\in\bbn} \rho_i(t)<0\quad \text{when}\quad\sup\limits_{i\in\bbn} \rho_i(t)=C_2.\]
                        Here, $C_1$ and $C_2$ are defined in \eqref{C-1-14} and \eqref{C-1-15}.
						\item Finally, we combine the  Gr\"onwall inequality with \eqref{C-1-3} to derive the quantitative estimates.
			\end{itemize}
	With the above four steps, given $\rho_0\in {\rm Int}(\mathcal{M})\cap l^{\infty}(\bbr)$, there exists a unique solution 
	$\rho(t) : [0, \infty)\to {\rm Int}(\mathcal{M})\cap l^{\infty}(\bbr)$
to FPE \eqref{FP} with initial value $\rho_0$, and we can find a compact subset $\mathcal{C}$ of $l_{\bbm}^2(\bbr)$ such that $\left\{\rho(t): ~t\in[0,+\infty)\right\}\subset \mathcal{C}$.  For $t\in[0,+\infty)$, we have 
\begin{align}\label{GF}
	\frac{\di \mathcal{F}(\rho | \bbm)}{\di t}=\left\langle 	\mathrm{diff}\mathcal{F}(
\rho| \bbm) ~,~	\frac{\di \rho}{\di t}\right\rangle_{l_{\bbm}^2}=-	g_{\rho}^{\Phi}\left(\frac{\di \rho}{\di t}, \frac{\di \rho}{\di t}\right)\le0,
\end{align}
and therefore 
\begin{align*}
		\frac{\di \mathcal{F}(\rho| \bbm)}{\di t}=0\quad\text{if and only if }\quad \frac{\di \rho}{\di t}=0.
\end{align*}
This equivalents to $\rho=\rho^{*}$ by (2) of Theorem \ref{T2.1}.
Now, we provide our proof in the following propositions.
	
	\end{proof}

\begin{proposition}\label{P2.1}
	Let \( \mathcal{G} = (\mathcal{V}, \mathcal{E},\bbm) \) be an infinite graph with vertex set \( \mathcal{V} = \{a_1, \dots, a_N, \dots\} \), edge set \( \mathcal{E} \), weight function $\bbm$, a potential \( \Phi = (\Phi_i)_{i=1}^{\infty}\) defined on \( \mathcal{V} \), and constant \( \beta \geq 0 \).  For any $t\ge0$ such that \[ 0<\inf\limits_{i\in\bbn} \rho_i(t)\le \sup\limits_{i\in\bbn} \rho_i(t)<\infty,\] we have 
		\begin{align}\label{C-1-6}
\frac{\di L(t)}{\di t}\le-\frac{\beta}{2}\frac{\left(\inf_i\rho_i(t)\right)}{\left(\sup_i \rho_i(t)\right)}\frac{\left(\inf_i \rho_i^{*}\right)}{\left(\sup_i \rho_i^{*}\right)}L(t), \quad\text{where}\quad 	L(t)=\sum\limits_{i=1}^{\infty}m_i\frac{(\rho_i(t)-\rho_i^{*})^{2}}{\rho_i^{*}}.
	\end{align}
\end{proposition}
\begin{proof}
First, we have 
\begin{align*}
		\frac{\di L(t)}{\di t}&=2\sum\limits_{i=1}^{\infty}m_i\frac{\rho_i(t)-\rho_i^{*}}{\rho_i^{*}}\frac{\di \rho_i(t)}{\di t}\\
		&=2\sum\limits_{i=1}^{\infty}m_i\frac{\rho_i(t)-\rho_i^{*}}{\rho_i^{*}}\Bigg[
\sum_{ \Phi_j > \Phi_i} m_j\left[ (\Phi_j + \beta \log \rho_j) -(\Phi_i + \beta \log \rho_i) \right] \rho_j \\
&\quad+ \sum_{ \Phi_j < \Phi_i} m_j\left[ (\Phi_j + \beta \log \rho_j) - (\Phi_i + \beta \log \rho_i) \right] \rho_i + \sum_{\Phi_j = \Phi_i}m_j \beta(\rho_j - \rho_i)\Bigg].	
\end{align*}
For simplicity of notation, throughout this proof we set
\[
\Gamma_i(t):=\frac{\rho_i(t)-\rho_i^\ast}{\rho_i^\ast},
\qquad i\in\bbn.
\]
Note that $\Phi_j-\Phi_i=\beta \log \rho_i^{*} -\beta \log \rho_j^{*}$ and $\rho_i^{*}=\rho_j^{*}$ when $\Phi_j=\Phi_i$.
Then, using this equality, we obtain
\begin{align*}
	\frac{\di L(t)}{\di t}&=2\sum\limits_{i=1}^{\infty}m_i\frac{\rho_i(t)-\rho_i^{*}}{\rho_i^{*}}\Bigg[
	\sum_{ \Phi_j > \Phi_i} m_j\left[ \beta \log \rho_i^{*} -\beta \log \rho_j^{*}+ \beta \log \rho_j- \beta \log \rho_i\right] \rho_j \\
	&\quad+ \sum_{ \Phi_j < \Phi_i} m_j\left[ \beta \log \rho_i^{*} -\beta \log \rho_j^{*}+ \beta \log \rho_j- \beta \log \rho_i \right] \rho_i + \sum_{\Phi_j = \Phi_i}m_j \beta\left(\frac{\rho_j}{\rho_j^{*}} - \frac{\rho_i}{\rho_i^{*}}\right)\frac{\rho_j^{*}+\rho_i^{*}}{2}\Bigg]\\
	&=2\sum\limits_{i=1}^{\infty}m_i\frac{\rho_i(t)-\rho_i^{*}}{\rho_i^{*}}\Bigg[
	\sum_{ \Phi_j > \Phi_i} m_j\left[ \beta \log \frac{\rho_j}{\rho_j^{*}}- \beta \log \frac{\rho_i}{\rho_i^{*}}\right] \rho_j \\
	&\quad+ \sum_{ \Phi_j < \Phi_i} m_j\left[ \beta \log \frac{\rho_j}{\rho_j^{*}}- \beta \log \frac{\rho_i}{\rho_i^{*}}\right] \rho_i + \sum_{\Phi_j = \Phi_i}m_j \beta\left(\frac{\rho_j}{\rho_j^{*}} - \frac{\rho_i}{\rho_i^{*}}\right)\frac{\rho_j^{*}+\rho_i^{*}}{2}\Bigg]\\
	&=2\sum_{i=1}^{\infty}m_i\Gamma_i(t)\Bigg[
	\sum_{\Phi_j>\Phi_i}m_j
	\left[
	\beta\log(1+\Gamma_j(t))
	-
	\beta\log(1+\Gamma_i(t))
	\right]\rho_j \\
	&\quad+
	\sum_{\Phi_j<\Phi_i}m_j
	\left[
	\beta\log(1+\Gamma_j(t))
	-
	\beta\log(1+\Gamma_i(t))
	\right]\rho_i +
	\sum_{\Phi_j=\Phi_i}m_j\beta
	\left(\Gamma_j(t)-\Gamma_i(t)\right)
	\frac{\rho_j^\ast+\rho_i^\ast}{2}
	\Bigg].
\end{align*} 
We combine the above equality with 
\begin{align*}
\min\left\{\frac{1}{1+\Gamma_j(t)},\frac{1}{1+\Gamma_i(t)}\right\}\le\frac{\log(1+\Gamma_j(t))-\log(1+\Gamma_i(t))}{\Gamma_j(t)-\Gamma_i(t)}\le \max\left\{\frac{1}{1+\Gamma_j(t)},\frac{1}{1+\Gamma_i(t)}\right\}
\end{align*}
and \(1+\Gamma_j(t)=\frac{\rho_j(t)}{\rho_j^{*}}\) to see 
\begin{align*}
		\frac{\di L(t)}{\di t}
		&\le -2\beta\sum\limits_{i=1}^{\infty}m_i
		\sum_{ \Phi_j > \Phi_i} m_j\left(\Gamma_j(t)-\Gamma_i(t)\right)^2 \min\left\{\frac{1}{1+\Gamma_j(t)},\frac{1}{1+\Gamma_i(t)}\right\}\rho_j \\
		&\quad- \beta\sum\limits_{i=1}^{\infty}m_i\sum_{\Phi_j = \Phi_i}m_j \left(\Gamma_j(t)-\Gamma_i(t)\right)^2\frac{\rho_j^{*}+\rho_i^{*}}{2}\\
&=-2\beta\sum\limits_{i=1}^{\infty}m_i
		\sum_{ \Phi_j > \Phi_i} m_j\left(\Gamma_j(t)-\Gamma_i(t)\right)^2 \min\left\{\frac{\rho_j^{*}}{\rho_j(t)},\frac{\rho_i^{*}}{\rho_i(t)}\right\}\rho_j \\
		&\quad-\beta \sum\limits_{i=1}^{\infty}m_i\sum_{\Phi_j = \Phi_i}m_j \left(\Gamma_j(t)-\Gamma_i(t)\right)^2\frac{\rho_j^{*}+\rho_i^{*}}{2}\\
&\le-\beta\frac{\left(\inf_i\rho_i(t)\right)}{\left(\sup_i \rho_i(t)\right)}\left(\inf_i \rho_i^{*}\right) \sum\limits_{i=1}^{\infty}m_i\sum_{\Phi_j \ge \Phi_i}m_j\left(\Gamma_j(t)-\Gamma_i(t)\right)^2 \\
		&\le-\frac{\beta}{2}\frac{\left(\inf_i\rho_i(t)\right)}{\left(\sup_i \rho_i(t)\right)}\left(\inf_i \rho_i^{*}\right) \sum\limits_{i=1}^{\infty}m_i\sum\limits_{j=1}^{\infty}m_j\left(\Gamma_j(t)-\Gamma_i(t)\right)^2.	
\end{align*}
Furthermore, we have
\begin{align*}
&\sum\limits_{i=1}^{\infty}m_i\sum\limits_{j=1}^{\infty}m_j\left(\Gamma_j(t)-\Gamma_i(t)\right)^2\ge\frac{1}{\sup\rho_i^{*}}	\sum\limits_{i=1}^{\infty}m_i\sum\limits_{j=1}^{\infty}m_j\left(\Gamma_j(t)-\Gamma_i(t)\right)^2\rho_i^{*}\\
	&\qquad	=\frac{1}{\sup\rho_i^{*}}\Bigg[	\sum\limits_{i=1}^{\infty}m_i\Gamma_i(t)^2\rho_i^{*}+\sum\limits_{j=1}^{\infty}m_j\Gamma_j(t)^2\rho_i^{*}+2\left(\sum\limits_{j=1}^{\infty}m_j\Gamma_j(t)\right)\left(\sum\limits_{i=1}^{\infty}m_i\Gamma_i(t)\rho_i^{*}\right)\Bigg]\\
&\qquad	\ge \frac{1}{\sup\rho_i^{*}}\Bigg[	\sum\limits_{i=1}^{\infty}m_i\Gamma_i(t)^2\rho_i^{*}+2\left(\sum\limits_{j=1}^{\infty}m_j\Gamma_j(t)\right)\left(\sum\limits_{i=1}^{\infty}m_i\Gamma_i(t)\rho_i^{*}\right)\Bigg]\\
&\qquad= \frac{1}{\sup\rho_i^{*}}L(t).
\end{align*}
Here, we used 
\[\sum\limits_{i=1}^{\infty}m_i\Gamma_i(t)^2\rho_i^{*}=\sum\limits_{i=1}^{\infty}m_i\frac{\left(\rho_i(t)-\rho_i^{*}\right)^2}{\rho_i^{*}}=L(t)  \]
and 
\[\sum\limits_{i=1}^{\infty}m_i\Gamma_i(t)\rho_i^{*}=\sum\limits_{i=1}^{\infty}m_i\rho_i(t)-\sum\limits_{i=1}^{\infty}m_i\rho_i^{*}=0.\]
Finally, we obtain 
\begin{align*}
	\begin{aligned}
		\frac{\di L(t)}{\di t}
		&\le-\frac{\beta}{2}\frac{\left(\inf_i\rho_i(t)\right)}{\left(\sup_i \rho_i(t)\right)}\left(\inf_i \rho_i^{*}\right) \sum\limits_{i=1}^{\infty}m_i\sum\limits_{j=1}^{\infty}m_j\left(\Gamma_j(t)-\Gamma_i(t)\right)^2\\
		&\le-\frac{\beta}{2}\frac{\left(\inf_i\rho_i(t)\right)}{\left(\sup_i \rho_i(t)\right)}\frac{\left(\inf_i \rho_i^{*}\right)}{\left(\sup_i \rho_i^{*}\right)}L(t).	\end{aligned}
\end{align*}
The proof is completed.
\end{proof}
Next, we prove the second step in the proof of (3) in Theorem \ref{T2.1}.
\begin{proposition}\label{P2.2}
	Let \( \mathcal{G} = (\mathcal{V}, \mathcal{E},\bbm) \) be an infinite graph with vertex set \( \mathcal{V} = \{a_1, \dots, a_N, \dots\} \), edge set \( \mathcal{E} \), weight function $\bbm$, a potential \( \Phi = (\Phi_i)_{i=1}^{\infty}\) defined on \( \mathcal{V} \), and constant \( \beta \geq 0 \). Define \[\tau=\sup\left\{t~\Big|~\forall~ s\in[0,t),~ 0<C_1\le\inf\limits_{i\in\bbn} \rho_i(s)\le \sup\limits_{i\in\bbn} \rho_i(s)\le C_2<\infty \right\},\]
   where $C_1$ and $C_2$ are defined in \eqref{C-1-14} and \eqref{C-1-15}. Then there exists $C_0$ (independent of $\tau$) dependent only on $\rho^{*}$ and $\rho_0$ such that 
		\begin{align}\label{C-1-7}
		\sup\limits_{0\le t\le \tau}	\|\rho(t)\|_{l_{\rm \bbm}^2}^2\le C_0.
		\end{align}
\end{proposition}
\begin{proof}
	Note that $0<\inf_i \rho_i^{*}\le \sup_i \rho_i^{*}<\infty$. Therefore, we have 
	\begin{align}\label{C-1-8}
	 \frac{1}{\sup_i \rho_i^{*}}\|\rho(t)-\rho^{*}\|_{l_{\rm \bbm}^2}^2\le	L(t)\le \frac{1}{\inf_i \rho_i^{*}}\|\rho(t)-\rho^{*}\|_{l_{\rm \bbm}^2}^2.
	\end{align}
	Moreover, we combine Proposition \ref{P2.1} to derive
	\begin{align*}
		\begin{aligned}
		\|\rho(t)\|_{l_{\rm \bbm}^2}^2&\le 2\|\rho(t)-\rho^{*}\|_{l_{\rm \bbm}^2}^2+2\|\rho^{*}\|_{l_{\rm \bbm}^2}^2\le2\left(\sup_i \rho_i^{*}\right)	L(t)+2\|\rho^{*}\|_{l_{\rm \bbm}^2}^2\\
		&\le 2\left(\sup_i \rho_i^{*}\right)	L(0)+2\|\rho^{*}\|_{l_{\rm \bbm}^2}^2\\
        &=2\left(\sup_i \rho_i^{*}\right)	\sum\limits_{i=1}^{\infty}m_i\frac{(\rho_i(0)-\rho_i^{*})^{2}}{\rho_i^{*}}+2\|\rho^{*}\|_{l_{\rm \bbm}^2}^2\\
        &\le4\left(\frac{\sup_i \rho_i^{*}}{\inf_i \rho_i^{*}}\right)\left(\sup_i\rho_i(0)\right)^2+4\sup_i \rho_i^{*}+2\|\rho^{*}\|_{l_{\rm \bbm}^2}^2=:C_0.
				\end{aligned}
		 \end{align*}
\end{proof}
\begin{remark}\label{R2.1}
For any small $1>\delta>0$, by Proposition \ref{P2.2}, there exists sufficiently large $N_{0}$ such that 
\begin{align*}
	\sum\limits_{i=N_{0}+1}^{\infty} m_i\rho_i\le \left(	\sum\limits_{i=N_{0}+1}^{\infty} m_i\right)^{\frac{1}{2}}	\left(\sum\limits_{i=N_{0}+1}^{\infty} m_i|\rho_i|^2\right)^{\frac{1}{2}}\le \left(	\sum\limits_{i=N_{0}+1}^{\infty} m_i\right)^{\frac{1}{2}} C_0<\delta.
\end{align*}
This implies 
\begin{align}\label{C-1-9}
	\sum\limits_{i=1}^{N_{0}} m_i\rho_i\ge 1-\delta.
	\end{align}
Moreover, for any $0\le t\le\tau$, we obtain that there exists some time-varying index $j_0(t)\le N_{0}$ such that \[\rho_{j_0(t)}(t)>1-\delta>0.\]
Otherwise, we have $\rho_{i}(t)\le1-\delta$ for all $i\le N_{0}$, 
\[\sum\limits_{i=1}^{N_{0}} m_i\rho_i\le\left( \sum\limits_{i=1}^{N_{0}} m_i\right)\left(1-\delta\right)<1-\delta.\]
This is contradictory to \eqref{C-1-9}. Note that the existence of $N_0$ depends only on $\delta$ and $C_0$, this fact will be used in the proof of Theorem \ref{T4.1} later.
\end{remark}
Before we prove the third step, we provide the following lemma to estimate the derivative of the infimum and supremum of solutions, which is similar to [Theorem 3.3. \cite{H-W-X-2025}].
\begin{lemma}\label{L2.1}
	Let \(\{f_n \}\) be a sequence of real-valued functions defined on $[0,\infty )$ which has the pointwise limit $f$ and satisfies the following conditions:
	\begin{enumerate}
		\item $f_n$ is Lipschitz continuous, and $\frac{\di}{\di t} f_n(t)$ is uniformly bounded with respect to $t$ and $n$, i.e., \\
		\begin{equation*}
			\exists~ C>0.\quad s.t. \quad \left|\frac{\di}{\di t} f_n(t)\right| < C, \quad \forall ~n\in \bbn, \ a.e. \ t\ge0 .
		\end{equation*}
		\item $f$ is almost everywhere differentiable.
	\end{enumerate} 
	Then, the following assertion holds:
	\begin{align}\label{C-1-10}
		\liminf_{n\rightarrow \infty}\frac{\di}{\di t}f_n(t)\le \frac{\di}{\di t} f(t) \le \limsup_{n\rightarrow \infty}\frac{\di}{\di t}f_n(t)\quad a.e.\ t\ge 0.\end{align}
\end{lemma}
\begin{proof}
	Clearly, $\limsup_{n\rightarrow \infty}\frac{\di}{\di t}f_n(t)$ is also uniformly bounded by the same constant $C$ in condition (1). Hence, both $\frac{\di}{\di t}f_n(t)$ and $\limsup_{n\rightarrow \infty}\frac{\di}{\di t}f_n(t)$ are locally integrable, and the set of non-Lebesgue points has zero measure. More precisely, we have
	\begin{equation}\label{Lebesgue}
		\begin{cases}
			\displaystyle\lim_{h\rightarrow 0} \frac{1}{h} \int_{t-h}^{t}\frac{\di}{\di s}{f_n(s)}\di s = \frac{\di}{\di t}{f_n (t)},\vspace{6pt}\\
			\displaystyle\lim_{h\rightarrow 0} \frac{1}{h} \int_{t-h}^{t}\limsup_{n\rightarrow \infty}\frac{\di }{\di s}{f_n(s)}\di s = \limsup_{n \rightarrow \infty}\frac{\di}{\di t}{f_n(t)}, \quad a.e. \ t> 0.
		\end{cases}
	\end{equation}
	Suppose that the contrary holds, then there exists a positive measure set $\mathcal{S}\subset [0, \infty)$ such that
	\[\frac{\di}{\di t} f(t) > \limsup_{n\rightarrow \infty}\frac{\di}{\di t}f_n(t), \quad \forall ~t \in \mathcal{S}.\]
	Now, we fix $t_0\in \mathcal{S}$ such that $f, f_n$ are differentiable at $t_0$ and satisfies (\ref{Lebesgue}). Then, we have
	\begin{align*}
		\frac{\di}{\di t} f(t_0) &= \lim_{h\rightarrow 0}\lim_{n \rightarrow \infty }\frac{f_n(t_0-h) - f(t_0)}{h} = \lim_{h\rightarrow 0}\lim_{n  \rightarrow \infty } \frac{1}{h} {\int_{t_0 - h}^{t_0} \frac{\di}{\di s}f_n(s)\di s} \\
		&= \lim_{h\rightarrow 0}\limsup_{n \rightarrow \infty }\frac{1}{h}{\int_{t_0 - h}^{t_0} \frac{\di}{\di s}f_n(s)\di s} \le \lim_{h\rightarrow 0} \frac{1}{h}{\int_{t_0 - h}^{t_0} \limsup_{n \rightarrow \infty }\frac{\di}{\di s}f_n(s)\di s} \\
		&=\limsup_{n \rightarrow \infty} \frac{\di}{\di t}{f_n(t_0)}.
	\end{align*}
This yields a contradiction. In the inequality in the second-to-last line, we used the reverse Fatou lemma. This concludes the proof of the right side of \eqref{C-1-10}. The proof of the left side of \eqref{C-1-10} is similar to the right side; we omit the details here.
\end{proof}
Now, we show that the solution will stay in the invariant set $\mathcal{C}$ of $l_{\bbm}^2(\bbr)$. We fix $\delta=\frac{1}{2}$ in Remark \ref{R2.1} and obtain $N_0$ such that there exists some density $j_0(t)\le N_{0}$ such that \begin{equation}\label{P_inf}\rho_{j_0(t)}(t)>1-\delta=\frac{1}{2}.\end{equation}
To simplify the notation, we drop the dependence of $j_0(t)$ on $t$ and denote it by $j_0$. And we denote \begin{align}\label{M_N0}M_{N_0}:=\min\limits_{i\in[N_0]}m_i,\end{align} where the finite index set $[N_0]:=\left\{1,2,...,N_0\right\}$. Now, we define the constants $C_1$ and $C_2$ as follows.
\begin{align}\label{C-1-14}
	C_1:=\min\left\{\frac{1}{2}\inf\limits_{i\in\bbn}\rho_{i}(0),~~\frac{e^{-\frac{4\|\Phi\|_{l^\infty}}{M_{N_0}\beta}}}{4},~~ \frac{\beta M_{N_0}}{4(2\|\Phi\|_{l^\infty}+\beta)}\right\},
\end{align}
and
\begin{align}\label{C-1-15}
	C_2:=\max\left\{2\sup\limits_{i\in\bbn}\rho_{i}(0),\quad \frac{e^{\frac{2\|\Phi\|_{l^\infty}}{\beta C_1M_{N_0}}}}{2M_{N_0}},~~\frac{2\|\Phi\|_{l^\infty}+2\beta}{\beta M_{N_0}},~\frac{2}{M_{N_0}}\right\}.
\end{align}
\begin{proposition}\label{P2.3}
		Let \( \mathcal{G} = (\mathcal{V}, \mathcal{E},\bbm) \) be an infinite graph with vertex set \( \mathcal{V} = \{a_1, \dots, a_N, \dots\} \), edge set \( \mathcal{E} \), weight function $\bbm$, a potential \( \Phi = (\Phi_i)_{i=1}^{\infty}\in l^{\infty}(\bbr) \) defined on \( \mathcal{V} \), and constant \( \beta> 0 \). Suppose that initial value \( \rho(0) = \rho_0 \in {\rm Int} ( \mathcal{M} )\cap l^{\infty}(\bbr) \), then we obtain the solutions will stay in the invariant set $\mathcal{C}$ of $l_{\bbm}^2(\bbr)$, where 
		\begin{align}\label{C-1-12}
			\mathcal{C}:=\left\{\rho\in \mathcal{M}~~\Big|~~0<C_1\le\inf\limits_{i\in\bbn} \rho_i\le \sup\limits_{i\in\bbn} \rho_i\le C_2<\infty\right\}.
		\end{align}
\end{proposition}
\begin{proof}

Define \[\tau=\sup\left\{t~\Bigg|~\forall~ s\in[0,t),~ 0<C_1\le\inf\limits_{i\in\bbn} \rho_i(s)\le \sup\limits_{i\in\bbn} \rho_i(s)\le C_2<\infty \right\}.\]
Obviously we have $\tau>0$ by the definition of $C_1$ and $C_2$ in \eqref{C-1-14}-\eqref{C-1-15}. We want to prove that $\tau=\infty$. Suppose that the contrary holds, we have $\tau<\infty$. Then, we obtain \begin{align}\label{Contradiction}
C_1=\inf\limits_{i\in\bbn} \rho_i(\tau)\quad \text{or}\quad  \sup\limits_{i\in\bbn} \rho_i(\tau)=C_2.
\end{align}
Note that the derivative of $\rho_i(t)$ is uniformly bounded when we have $C_1\le\inf\limits_{i\in\bbn} \rho_i(t) \le \sup\limits_{i\in\bbn} \rho_i(t)\le C_2$. More precisely, we use \eqref{FP} to see
\begin{align*}
	\begin{aligned}
		\frac{\di\rho_i}{\di t}&  = 
	\sum_{ \Phi_j > \Phi_i} m_j\left[ (\Phi_j + \beta \log \rho_j) -(\Phi_i + \beta \log \rho_i) \right] \rho_j \\
	&	\quad+  \sum_{ \Phi_j < \Phi_i} m_j\left[ (\Phi_j + \beta \log \rho_j) - (\Phi_i + \beta \log \rho_i) \right] \rho_i + \sum_{\Phi_j = \Phi_i}m_j \beta(\rho_j - \rho_i)\\
	&\le \Bigg(2\|\Phi\|_{l^\infty}+\log\left(\sup\limits_{i\in\bbn} \rho_i(t)\right)-\log\left(\inf\limits_{i\in\bbn} \rho_i(t)\right)+\beta\Bigg)\left(\sup\limits_{i\in\bbn} \rho_i(t)\right)
		\end{aligned}
\end{align*}
and 
\begin{align*}
	\begin{aligned}
		\frac{\di\rho_i}{\di t}&  = 
		\sum_{ \Phi_j > \Phi_i} m_j\left[ (\Phi_j + \beta \log \rho_j) -(\Phi_i + \beta \log \rho_i) \right] \rho_j \\
		&	\quad+  \sum_{ \Phi_j < \Phi_i} m_j\left[ (\Phi_j + \beta \log \rho_j) - (\Phi_i + \beta \log \rho_i) \right] \rho_i + \sum_{\Phi_j = \Phi_i}m_j \beta(\rho_j - \rho_i)\\
		&\ge- \Bigg(2\|\Phi\|_{l^\infty}+\log\left(\sup\limits_{i\in\bbn} \rho_i(t)\right)-\log\left(\inf\limits_{i\in\bbn} \rho_i(t)\right)+\beta\Bigg)\left(\sup\limits_{i\in\bbn} \rho_i(t)\right).
	\end{aligned}
\end{align*}
This implies that \[\sup\limits_{i\in[N]} \rho_i(t),\quad \inf\limits_{i\in[N]} \rho_i(t), \quad \sup\limits_{i\in\bbn} \rho_i(t), \quad \text{and} \quad\inf\limits_{i\in\bbn} \rho_i(t)\] are Lipschitz continuous, therefore we can apply Lemma \ref{L2.1} to see 
\begin{align}\label{C-1-13}
		\liminf_{N\rightarrow \infty}\frac{\di}{\di t}\min\limits_{i\in[N]} \rho_i(t)\le \frac{\di}{\di t}\inf\limits_{i\in\bbn} \rho_i(t)\ \ \ \text{and}\ \ \ \frac{\di}{\di t}\sup\limits_{i\in\bbn} \rho_i(t)\le \limsup_{N\rightarrow \infty}\frac{\di}{\di t}\max\limits_{i\in[N]} \rho_i(t)\quad a.e.\ 0\le t\le \tau.
\end{align}
To complete the proof, we only need to show that \[ \frac{\di}{\di t}\inf\limits_{i\in\bbn} \rho_i(t)>0\quad \text{when} \quad\rho_{\inf}(t):=\inf\limits_{i\in\bbn} \rho_i(t)=C_1,\]\[~ \frac{\di}{\di t}\sup\limits_{i\in\bbn} \rho_i(t)<0\quad \text{when}\quad\rho_{\sup}(t):=\sup\limits_{i\in\bbn} \rho_i(t)=C_2.\]
Now, we prove the above inequalities one by one.
\vspace{0.2cm}

\noindent$\bullet$ Step A (The derivative of infimum):
Note that
\[\lim\limits_{N\to\infty}\min\limits_{i\in[N]} \rho_i(t)=\inf\limits_{i\in\bbn} \rho_i(t),\quad\lim\limits_{N\to\infty}\log\left(\min\limits_{i\in[N]} \rho_i(t)\right)=\log\left(\inf\limits_{i\in\bbn} \rho_i(t)\right). \]
Then, for any $\varepsilon$, there exists some $N\ge N_0$ such that 
\[\min\limits_{i\in[N]} \rho_i(t)<\inf\limits_{i\in\bbn} \rho_i(t)+\varepsilon,\quad\log\left(\min\limits_{i\in[N]} \rho_i(t)\right)<\log\left(\inf\limits_{i\in\bbn} \rho_i(t)\right)+\varepsilon. \]
 For simplicity of notation, we drop time $t$ and choose a time-varying index $i_N$ such that 
\begin{align*}
	\begin{aligned}
	\frac{\di  \rho_{i_N}}{\di t}&=	\frac{\di \min\limits_{i\in[N]} \rho_i}{\di t}\\
		&  = 
		\sum_{ \Phi_j > \Phi_{i_N}} m_j\left[ (\Phi_j + \beta \log \rho_j) -(\Phi_{i_N} + \beta \log \rho_{i_N}) \right] \rho_j \\
		&\quad+  \sum_{ \Phi_j < \Phi_{i_N}} m_j\left[ (\Phi_j + \beta \log \rho_j) - (\Phi_{i_N} + \beta \log \rho_{i_N}) \right] \rho_{i_N} + \sum_{\Phi_j = \Phi_i}m_j \beta(\rho_j - \rho_{i_N})\\
		&\ge \sum_{ \Phi_j > \Phi_{i_N}} m_j\left[ (\Phi_j + \beta \log \rho_j) -(\Phi_{i_N} + \beta \log \rho_{\inf}) \right] \rho_j \\
		&\quad+  \sum_{ \Phi_j < \Phi_{i_N}} m_j\left[ (\Phi_j + \beta \log \rho_j) - (\Phi_{i_N} + \beta \log \rho_{\inf}) \right] \rho_{\inf} + \sum_{\Phi_j = \Phi_i}m_j \beta(\rho_j - \rho_{\inf})-\mathcal{O}(\varepsilon)\\
		&=:\mathcal{I}_1+\mathcal{I}_2+\mathcal{I}_3-\mathcal{O}(\varepsilon).
	\end{aligned}
\end{align*}
We emphasize that $C_1<\frac{1}{4}$, therefore $j_0\neq i_N$ for sufficiently large $N$. Next, we further divide into three cases to analyze the derivative of $\inf\limits_{i\in\bbn} \rho_i(t).$

\noindent$\bullet$ Case A.1 ($\Phi_{j_0} > \Phi_{i_N}$): We use 
\[ \log \rho_j\ge\log \rho_{\inf} \] and the uniform control of the density from below at a given vertex in the graph ensured by \eqref{P_inf} to deduce
\begin{align}\label{C-1-16}
	\begin{aligned}
	\mathcal{I}_1&=\sum_{ \Phi_j > \Phi_{i_N}} m_j\left[ (\Phi_j + \beta \log \rho_j) -(\Phi_{i_N} + \beta \log \rho_{\inf}) \right] \rho_j\\
&	\ge m_{j_0}\left[  \beta \log \rho_{j_0}-  \beta \log \rho_{\inf} \right] \rho_{j_0}\\
&\ge M_{N_0}\left[  \beta \log \frac{1}{2}-  \beta \log \rho_{\inf} \right] \frac{1}{2}=\log\left(\left(2\rho_{\inf}\right)^{-\frac{\beta M_{N_0}}{2}}\right).
\end{aligned}
\end{align}
Next, we use $\rho_{\inf}=C_1<1$, $\log \rho_j\ge\log \rho_{\inf}$, and $\Phi_j-\Phi_i\ge -2\|\Phi\|_{l^\infty}$ due to our assumptions to find 
\begin{align}\label{C-1-17}
	\begin{aligned}
		\mathcal{I}_2&= \sum_{ \Phi_j < \Phi_{i_N}} m_j\left[ (\Phi_j + \beta \log \rho_j) - (\Phi_{i_N} + \beta \log \rho_{\inf}) \right] \rho_{\inf} \ge -2\rho_{\inf}\|\Phi\|_{l^\infty}\ge -2\|\Phi\|_{l^\infty}.
	\end{aligned}
\end{align}
Moreover, we use $\rho_j\ge\rho_{\inf}$ to get $\mathcal{I}_3\ge0$. Furthermore, we combine \eqref{C-1-14}, \eqref{C-1-16} and \eqref{C-1-17} to derive 
\begin{align*}
	\begin{aligned}
		\frac{\di  \rho_{i_N}}{\di t}\ge\log\left(\left(2\rho_{\inf}\right)^{-\frac{\beta M_{N_0}}{2}}\right)-2\|\Phi\|_{l^\infty}-\mathcal{O}(\varepsilon)\ge \frac{\beta M_{N_0}}{2}\log2- \mathcal{O}(\varepsilon).
	\end{aligned}
\end{align*}	
Finally, we combine the above inequality and \eqref{C-1-13} to see
\[\frac{\di}{\di t}\inf\limits_{i\in\bbn} \rho_i(t)\ge	\liminf_{N\rightarrow \infty}\frac{\di}{\di t}\min\limits_{i\in[N]} \rho_i(t)\ge\frac{\beta M_{N_0}}{2}\log 2.\]
\vspace{0.2cm}

\noindent$\bullet$ Case A.2 ($\Phi_{j_0} <\Phi_{i_N}$): We use a similar argument in Case A.1 to see $\mathcal{I}_1
\ge0$ and $\mathcal{I}_3\ge0$. Now, we again use the uniform control of the density from below at a given node \eqref{P_inf} to estimate the term $\mathcal{I}_2$ as below.
\begin{align}\label{C-1-18}
	\begin{aligned}
		\mathcal{I}_2&= \sum_{ \Phi_j <\Phi_{i_N}} m_j\left[ (\Phi_j + \beta \log \rho_j) - (\Phi_{i_N} + \beta \log \rho_{\inf}) \right] \rho_{\inf} \\
		&	\ge \big(-2\|\Phi\|_{l^\infty} +m_{j_0}\left[  \beta \log \rho_{j_0}-  \beta \log \rho_{\inf} \right] \big)\rho_{\inf}\\
		&\ge \Bigg(-2\|\Phi\|_{l^\infty} +M_{N_0}\left[  \beta \log \frac{1}{2}-  \beta \log \rho_{\inf} \right] \Bigg)\rho_{\inf}\\
		&=\Big(-2\|\Phi\|_{l^\infty} +\log\left(\left(2\rho_{\inf}\right)^{-\beta M_{N_0}}\right) \Big)\rho_{\inf}\\
		&\ge\Big(2\|\Phi\|_{l^\infty}+\frac{\beta M_{N_0}}{2}\log 2  \Big)C_1.
	\end{aligned}
\end{align}
Finally, we combine the above fact and \eqref{C-1-13} to see 
\[\frac{\di}{\di t}\inf\limits_{i\in\bbn} \rho_i(t)\ge	\liminf_{N\rightarrow \infty}\frac{\di}{\di t}\min\limits_{i\in[N]} \rho_i(t)\ge\Big(2\|\Phi\|_{l^\infty}+\frac{\beta M_{N_0}}{2}\log 2  \Big)C_1.\]
\vspace{0.2cm}

\noindent$\bullet$ Case A.3 ($\Phi_{j_0} = \Phi_{i_N}$): First, we use the similar argument in \eqref{C-1-16} and \eqref{C-1-17} to see $\mathcal{I}_1\ge0$ and 
\begin{align*}
	\mathcal{I}_2\ge-2\|\Phi\|_{l^\infty}\rho_{\inf}.
\end{align*} 
Next, we make use again of the uniform control of the density from below at a given vertex \eqref{P_inf} to obtain
\begin{align*}
	\begin{aligned}
		\mathcal{I}_3&=\sum_{\Phi_j = \Phi_i}m_j \beta(\rho_j - \rho_{\inf})\ge m_{j_0} \beta\rho_{j_0} - \beta\left(\sum_{\Phi_j = \Phi_i}m_j \right) \rho_{\inf}\ge \frac{\beta M_{N_0}}{2}-\beta \rho_{\inf}.
			\end{aligned}
	\end{align*}
Then, we combine the above inequalities and \eqref{C-1-13} to see 
\[\frac{\di}{\di t}\inf\limits_{i\in\bbn} \rho_i(t)\ge	\liminf_{N\rightarrow \infty}\frac{\di}{\di t}\min\limits_{i\in[N]} \rho_i(t)\ge-2\|\Phi\|_{l^\infty}\rho_{\inf}+\frac{\beta M_{N_0}}{2}-\beta \rho_{\inf}\ge\frac{\beta M_{N_0}}{4}.\]
We combine Cases A.1, A.2, and A.3 to deduce that there exists a interval $[\tau^{-},\tau]$ $(\tau^{-}<\tau)$ such that for any differentiable point $s$ of $\inf\limits_{i\in\bbn} \rho_i(s)$ in this interval, 
\begin{align}\label{lower}
  \frac{\di}{\di s}  \inf\limits_{i\in\bbn} \rho_i(s)>0.
\end{align}
Finally, we have 
\[\inf\limits_{0\le t\le\tau}\inf\limits_{i\in\bbn} \rho_i(t)>C_1.\]
 
\vspace{0.2cm}
\noindent$\bullet$ Step B (The derivative of supremum) : We use $\sum_{i=1}^{\infty} m_i\rho_i=1$ to obtain
\begin{align}\label{P_sup}\rho_i\le \frac{1}{M_{N_0}}, \quad \forall ~ i\le N_0, \end{align}
by definition of $M_{N_0}$ in \eqref{M_N0}.
 For any $\varepsilon$, there exists some $N\ge N_0$ such that 
\[\max\limits_{i\in[N]} \rho_i(t)>\sup\limits_{i\in\bbn} \rho_i(t)-\varepsilon,\quad\log\left(\max\limits_{i\in[N]} \rho_i(t)\right)>\log\left(\sup\limits_{i\in\bbn} \rho_i(t)\right)-\varepsilon. \]
For simplicity of notation, we drop time $t$ and choose a time-varying index $i_N$ such that 
\begin{align*}
	\begin{aligned}
		\frac{\di  \rho_{i_N}}{\di t}&=	\frac{\di \max\limits_{i\in[N]} \rho_i}{\di t}\\
		&  = 
		\sum_{ \Phi_j > \Phi_{i_N}} m_j\left[ (\Phi_j + \beta \log \rho_j) -(\Phi_{i_N} + \beta \log \rho_{i_N}) \right] \rho_j \\
		&\quad+  \sum_{ \Phi_j < \Phi_{i_N}} m_j\left[ (\Phi_j + \beta \log \rho_j) - (\Phi_{i_N} + \beta \log \rho_{i_N}) \right] \rho_{i_N} + \sum_{\Phi_j = \Phi_i}m_j \beta(\rho_j - \rho_{i_N})\\
		&\le \sum_{ \Phi_j > \Phi_{i_N}} m_j\left[ (\Phi_j + \beta \log \rho_j) -(\Phi_{i_N} + \beta \log \rho_{\sup}) \right] \rho_j \\
		&\quad+  \sum_{ \Phi_j < \Phi_{i_N}} m_j\left[ (\Phi_j + \beta \log \rho_j) - (\Phi_{i_N} + \beta \log \rho_{\sup}) \right] \rho_{\sup} + \sum_{\Phi_j = \Phi_i}m_j \beta(\rho_j - \rho_{\sup})+\mathcal{O}(\varepsilon)\\
		&=:\mathcal{I}_4+\mathcal{I}_5+\mathcal{I}_6+\mathcal{O}(\varepsilon).
	\end{aligned}
\end{align*}
We emphasize that $C_2\ge \frac{2}{M_{N_0}}$, therefore index $i_N > N_0$ for sufficiently large $N$ by \eqref{P_sup}. Next, we similarly further divide into three cases to analyze the derivative of $\sup\limits_{i\in\bbn} \rho_i(t).$

\noindent$\bullet$ Case B.1 ($\Phi_{1} > \Phi_{i_N}$): We use 
\[ \log \rho_j\le\log \rho_{\sup} \quad \text{and}\quad \sum\limits_{i=1}^{\infty} m_i\rho_i=1\]
to see
\begin{align}\label{C-1-19}
		\mathcal{I}_4&=\sum_{ \Phi_j > \Phi_{i_N}} m_j\left[ (\Phi_j + \beta \log \rho_j) -(\Phi_{i_N} + \beta \log \rho_{\sup}) \right] \rho_j\nonumber\\
		&\le m_{1}\left[  \beta \log \rho_{1}-  \beta \log \rho_{\sup} \right]\rho_{1}+2\|\Phi\|_{l^\infty}\sum_{ \Phi_j > \Phi_{i_N}} m_j\rho_j\nonumber\\
		&	\le m_{1}\left[  \beta \log \rho_{1}-  \beta \log \rho_{\sup} \right] \rho_{1}+2\|\Phi\|_{l^\infty}\nonumber\\
		&\le M_{N_0}\left[  -\beta \log M_{N_0}-  \beta \log \rho_{\sup} \right] C_1+2\|\Phi\|_{l^\infty}\nonumber\\
        &=\log\left(\left(M_{N_0}\rho_{\sup}\right)^{-\beta M_{N_0}C_1}\right)+2\|\Phi\|_{l^\infty}.
\end{align}
Next, we use $\log \rho_j\le\log \rho_{\sup}$ to find 
\begin{align}\label{C-1-20}
	\begin{aligned}
		\mathcal{I}_5 =\sum_{ \Phi_j < \Phi_{i_N}} m_j\left[ (\Phi_j + \beta \log \rho_j) - (\Phi_{i_N} + \beta \log \rho_{\sup}) \right] \rho_{\sup}\le0.\\
	\end{aligned}
\end{align}
Moreover, we use $\rho_j\le\rho_{\sup}$ to get $\mathcal{I}_6\le0$. Furthermore, we combine \eqref{C-1-15}, \eqref{C-1-19} and \eqref{C-1-20} to derive 
\begin{align*}
	\begin{aligned}
		\frac{\di  \rho_{i_N}}{\di t}\le\log\left(\left(M_{N_0}\rho_{\sup}\right)^{-\beta M_{N_0}C_1}\right)+2\|\Phi\|_{l^\infty}+\mathcal{O}(\varepsilon)\le -C_1\beta M_{N_0}\log 2+ \mathcal{O}(\varepsilon).
	\end{aligned}
\end{align*}	
Finally, we combine the above inequality and \eqref{C-1-13} to see 
\[\frac{\di}{\di t}\sup\limits_{i\in\bbn} \rho_i(t)\le	\limsup_{N\rightarrow \infty}\frac{\di}{\di t}\max\limits_{i\in[N]} \rho_i(t)\le -C_1\beta M_{N_0}\log 2.\]
\vspace{0.2cm}

\noindent$\bullet$ Case B.2 ($\Phi_{1} <\Phi_{i_N}$): We use a similar argument in Case B.1 to see $\mathcal{I}_6\le0$. Now, we use  $\log \rho_j\le\log \rho_{\sup}$, $ \rho_{\sup}\ge C_1$, $\sum_{i=1}^{\infty} m_i\rho_i=1$, and $\rho_1\le\frac{1}{M_{N_0}}$ to estimate the terms $\mathcal{I}_4$ and  $\mathcal{I}_5$. That is, we have
\begin{align}\label{C-1-21}
	\begin{aligned}
		\mathcal{I}_4=\sum_{ \Phi_j > \Phi_{i_N}} m_j\left[ (\Phi_j + \beta \log \rho_j) -(\Phi_{i_N} + \beta \log \rho_{\sup}) \right] \rho_j\le2\|\Phi\|_{l^\infty}\sum_{ \Phi_j > \Phi_{i_N}} m_j\rho_j\le2\|\Phi\|_{l^\infty},
	\end{aligned}
\end{align}
and
\begin{align}\label{C-1-22}
	\begin{aligned}
		\mathcal{I}_5&= \sum_{ \Phi_j <\Phi_{i_N}} m_j\left[ (\Phi_j + \beta \log \rho_j) - (\Phi_{i_N} + \beta \log \rho_{\sup}) \right] \rho_{\sup} \\
		&	\le m_{1}\left[  \beta \log \rho_{1}-  \beta \log \rho_{\sup} \right] \rho_{\sup}\\
        	&\le M_{N_0}\left[  -\beta \log M_{N_0}-  \beta \log \rho_{\sup} \right] \rho_{\sup}\\
		&\le \log\left(\left(M_{N_0}\rho_{\sup}\right)^{-\beta M_{N_0}C_1}\right).
	\end{aligned}
\end{align}
Finally, we combine above inequalities, \eqref{C-1-15}, and \eqref{C-1-13} to find
\begin{align*}
	\begin{aligned}
		\frac{\di}{\di t}\sup\limits_{i\in\bbn} \rho_i(t)&\le	\limsup_{N\rightarrow \infty}\frac{\di}{\di t}\max\limits_{i\in[N]} \rho_i(t)\\
		&\le \log\left(\left(M_{N_0}\rho_{\sup}\right)^{-\beta M_{N_0}C_1}\right)+2\|\Phi\|_{l^\infty}\le-C_1\beta M_{N_0}\log 2.
	\end{aligned}
\end{align*}

\vspace{0.2cm}

\noindent$\bullet$ Case B.3 ($\Phi_{1} = \Phi_{i_N}$): First, we use the similar argument in \eqref{C-1-19} and \eqref{C-1-20} and $\sum_{i=1}^{\infty} m_i\rho_i=1$ to see 
\begin{align*}
	\mathcal{I}_4+\mathcal{I}_5\le2\|\Phi\|_{l^\infty}.
\end{align*} 
Next, we have
\begin{align*}
	\begin{aligned}
		\mathcal{I}_6&=\sum_{\Phi_j = \Phi_i}m_j \beta(\rho_j - \rho_{\sup})\le  \beta- \beta m_1  \rho_{\sup}\le \beta- \beta M_{N_0}  \rho_{\sup}.
	\end{aligned}
\end{align*}
Then, we combine the above inequalities, \eqref{C-1-15} and \eqref{C-1-13} to find 
\[\frac{\di}{\di t}\sup\limits_{i\in\bbn} \rho_i(t)\le	\limsup_{N\rightarrow \infty}\frac{\di}{\di t}\max\limits_{i\in[N]} \rho_i(t)\le 2\|\Phi\|_{l^\infty}+ \beta- \beta M_{N_0}  \rho_{\sup}\le-\beta.\]
Combining Cases B.1, B.2, and B.3, there exists a interval $[\tau^{-},\tau]$ $(\tau^{-}<\tau)$ such that for any differentiable point $s$ of $\sup\limits_{i\in\bbn} \rho_i(s)$ in this interval, 
\begin{align}\label{upper}
  \frac{\di}{\di s}  \sup\limits_{i\in\bbn} \rho_i(s)<0.
\end{align}
Then, we have
\[\sup\limits_{0\le t\le\tau}\sup\limits_{i\in\bbn} \rho_i(t)< C_2.\]
Finally, we combine Steps A and B to derive $\tau=\infty$. The proof is completed.
\end{proof}
Now, we are ready to provide the long-time behavior of FPE \eqref{FP} as follows.
\begin{proposition}\label{P2.4}
		Let \( \mathcal{G} = (\mathcal{V}, \mathcal{E},\bbm) \) be an infinite graph with vertex set \( \mathcal{V} = \{a_1, \dots, a_N, \dots\} \), edge set \( \mathcal{E} \), weight function $\bbm$, a potential \( \Phi = (\Phi_i)_{i=1}^{\infty}\in l^{\infty}(\bbr) \) defined on \( \mathcal{V} \), and constant \( \beta> 0 \). Suppose that initial value \( \rho(0) = \rho_0 \in {\rm Int} ( \mathcal{M} ) \cap l^{\infty}(\bbr)\), then there exists some constants $C$ depend only on $\beta$, $\rho^{*}$, and $\rho_0$ such that
			\begin{align}\label{C-1-23}
				L(t)\le L(0)e^{-Ct}.
				\end{align}
\end{proposition}
\begin{proof}
	We combine Propositions \ref{P2.1} and \ref{P2.3} to see 
\begin{align}\label{NewC_3}
\frac{\di L(t)}{\di t}\le-\frac{\beta}{2}\frac{\left(\inf_i\rho_i(t)\right)}{\left(\sup_i \rho_i(t)\right)}\frac{\left(\inf_i \rho_i^{*}\right)}{\left(\sup_i \rho_i^{*}\right)}L(t)\le -\frac{\beta}{2}\frac{C_1}{C_2}\frac{\left(\inf_i \rho_i^{*}\right)}{\left(\sup_i \rho_i^{*}\right)}L(t)=:-CL(t).
\end{align}	
Then, by the Gr\"onwall inequality, we finished the proof.
\end{proof}
\begin{remark}\label{R2.2}
We can use \eqref{C-1-8} to see
\[\|\rho(t)-\rho^{*}\|_{l_{\rm \bbm}^2}^2\le\left( \sup_i \rho_i^{*}\right)L(t)\le\left(\sup_i \rho_i^{*}\right)L(0)e^{-Ct}\le\frac{\sup_i \rho_i^{*}}{\inf_i \rho_i^{*}}\|\rho_0-\rho^{*}\|_{l_{\rm \bbm}^2}^2e^{-Ct} . \]
The proof of (3).(c) in Theorem \ref{T2.1} is completed. Moreover, we combine $m_i>0$ to see
\[\lim\limits_{t\to\infty}\rho_i(t)=\rho_i^{*},\quad \forall~ i\in\bbn.\]
\end{remark}
\begin{corollary}
	Let \( \mathcal{G}=(\mathcal{V},\mathcal{E},\bbm) \) be an infinite graph with vertex set
	\( \mathcal{V}=\{a_1,a_2,\dots\} \), edge set \( \mathcal{E} \), weight function \(\bbm\),
	a potential \( \Phi=(\Phi_i)_{i=1}^{\infty} \) defined on \( \mathcal{V} \), and constant
	\( \beta\ge0 \).
	\begin{enumerate}
		\item If \( \beta=0 \), then the FPE reduces to
		\[
		\frac{\di \rho_i}{\di t}
		=
		\sum_{\Phi_j>\Phi_i}m_j(\Phi_j-\Phi_i)\rho_j
		+
		\sum_{\Phi_j<\Phi_i}m_j(\Phi_j-\Phi_i)\rho_i,
		\qquad i\in\bbn.
		\]
		
		\item If the potential \(\Phi_i\) is constant for all \(i\in\bbn\) and \(\beta>0\), then the FPE becomes the linear discrete master equation
		\[
		\frac{\di\rho_i}{\di t}
		=
		\beta\sum_{j\in\bbn}m_j(\rho_j-\rho_i),
		\qquad i\in\bbn.
		\]
		In this case, for every initial value
		\[
		\rho_0\in \mathcal M\cap l_{\bbm}^2(\bbr),
		\]
		there exists a unique global solution
		\[
		\rho\in C([0,\infty);l_{\bbm}^2(\bbr))\cap C^1((0,\infty);l_{\bbm}^2(\bbr)).
		\]
		Moreover, the solution converges exponentially to the equilibrium
		\[
		\rho^c=(1,1,\dots)
		\]
		in the \(l_{\bbm}^2\)-metric. More precisely,
		\[
		\|\rho(t)-\rho^c\|_{l_{\bbm}^2}
		=
		e^{-\beta t}\|\rho_0-\rho^c\|_{l_{\bbm}^2},
		\qquad t\ge0.
		\]
		If, in addition, \(\rho_0\in l^\infty(\bbr)\), then the same convergence also holds in the \(l^\infty\)-metric:
		\[
		\|\rho(t)-\rho^c\|_{l^\infty}
		=
		e^{-\beta t}\|\rho_0-\rho^c\|_{l^\infty}.
		\]
	\end{enumerate}
\end{corollary}
\begin{proof}
	We only prove the second assertion. Since \(\Phi_i\) is constant, the FPE reduces to
	\[
	\frac{\di\rho_i}{\di t}
	=
	\beta\sum_{j\in\bbn}m_j(\rho_j-\rho_i).
	\]
	Using the mass constraint
	\[
	\sum_{j=1}^{\infty}m_j\rho_j(t)=1,
	\]
	we obtain
	\[
	\frac{\di\rho_i}{\di t}
	=
	\beta(1-\rho_i),
	\qquad i\in\bbn.
	\]
	Therefore,
	\[
	\rho_i(t)=1+e^{-\beta t}(\rho_i(0)-1),
	\qquad i\in\bbn.
	\]
	This formula immediately gives existence and uniqueness for every
	\(\rho_0\in\mathcal M\cap l_{\bbm}^2(\bbr)\). Moreover,
	\[
	\rho_i(t)-1=e^{-\beta t}(\rho_i(0)-1),
	\]
	and hence
	\[
	\|\rho(t)-\rho^c\|_{l_{\bbm}^2}
	=
	e^{-\beta t}\|\rho_0-\rho^c\|_{l_{\bbm}^2}.
	\]
	If \(\rho_0\in l^\infty(\bbr)\), the same explicit formula yields
	\[
	\|\rho(t)-\rho^c\|_{l^\infty}
	=
	e^{-\beta t}\|\rho_0-\rho^c\|_{l^\infty}.
	\]
	The proof is completed.
\end{proof}

\begin{remark}
The previous corollary shows that, in the linear constant-potential case, the assumption
\(\rho_0\in l^\infty(\bbr)\) can be removed. In this case, the equation reduces to a linear
master equation, and existence, uniqueness, and exponential convergence follow directly
for initial data in \(l_{\bbm}^2(\bbr)\). This is consistent with the situation in
\cite{E-2014}, where nonlocal diffusion equations associated with jump processes can be
treated under natural energy and integrability assumptions.

For the nonlinear FPE with a non-constant potential, however, removing the \(l^\infty\)-bound is substantially more delicate. The current well-posedness argument relies on the invariant region
	\[
	0<C_1\le \inf_{i\in\bbn}\rho_i(t)\le\sup_{i\in\bbn}\rho_i(t)\le C_2<\infty.
	\]
	Extending the theory to unbounded \(l_{\bbm}^2\)-densities in the fully nonlinear case would require additional compactness or moment estimates. We leave this as an interesting open problem for future research.
\end{remark}
	If we change the potential field $\Phi_i(t)$ to be
	\[\bar{\Phi}_i(t)=\Phi_i-\beta \log\rho_i(t),\quad \text {for }\quad i\in\bbn.\]
Then, we can use the new distance $g_{\rho}^{\bar{\Phi}}$ to obtain the following theorem.
\begin{theorem}\label{T2.2}
	Let \( \mathcal{G} = (\mathcal{V}, \mathcal{E},\bbm) \) be an infinite graph with vertex set \( \mathcal{V} = \{a_1, \dots, a_N, \dots\} \), edge set \( \mathcal{E} \), weight function $\bbm$, a potential \( \Phi = (\Phi_i)_{i=1}^{\infty}\in l^{\infty}(\bbr) \) defined on \( \mathcal{V} \), and a constant \( \beta \geq 0 \). Then, the following assertions hold:

\begin{enumerate}
	\item The gradient flow of \( \mathcal{F} \) on the Hilbert manifold \( ( \mathcal{M} , d^{\bar{\Phi}}) \) of probability densities on \( \mathcal{V} \) is given by the discrete FPE:
	\begin{align*}
		\frac{\di\rho_i}{\di t} = 
		& \sum_{ \bar{\Phi}_j > \bar{\Phi}_i} m_j\left[ (\Phi_j + \beta \log \rho_j) -(\Phi_i + \beta \log \rho_i) \right] \rho_j \\
		&	\quad+  \sum_{ \bar{\Phi}_j < \bar{\Phi}_i} m_j\left[ (\Phi_j + \beta \log \rho_j) - (\Phi_i + \beta \log \rho_i) \right] \rho_i.
	\end{align*}
	
	\item For all \( \beta > 0 \), the Gibbs distribution
	\[
	\rho_i^* = \frac{e^{\frac{-\Phi_i}{\beta}}}{\sum_i m_ie^{\frac{-\Phi_i}{\beta}}}, \quad i\in\bbn, 
	\]
	is the unique stationary distribution of the FPE. Furthermore, \( \mathcal{F} \) attains its global minimum at \( \rho^* \).
	
	\item For all \( \beta > 0 \), there exists a unique solution
	\[
	\rho(t): [0, \infty) \to {\rm Int} ( \mathcal{M} )\cap l^{\infty}(\bbr)
	\]
	to the FPE with initial value \( \rho(0) = \rho_0 \in {\rm Int} ( \mathcal{M} ) \cap l^{\infty}(\bbr)\), such that:
	
	\begin{itemize}
		\item[(a)] \( \mathcal{F}(
\rho|\bbm) \) is non-increasing in time \( t \);
\item[(b)]  there exists some constant $\bar{C}_1$ and $\bar{C}_2$ depend on initial value such that \[0<\bar{C}_1\le\inf\limits_{0\le t<\infty}\inf\limits_{i\in\bbn} \rho_i(t)\le \sup\limits_{0\le t<\infty}\sup\limits_{i\in\bbn} \rho_i(t)\le \bar{C}_2<\infty;\]
		\item[(c)] \( \rho(t)\) converges to \(\rho^* \) exponentially under the $l_{\rm \bbm}^2$-metric as \( t \to +\infty \).
	\end{itemize}
\end{enumerate}
\end{theorem}
\begin{proof}
The proofs are similar to Theorem \ref{T2.1}; we omit them here.
\end{proof}
\section{Talagrand-type inequalities}\label{sec:4}
	\setcounter{equation}{0}
Talagrand-type inequalities have been studied widely in Euclidean space and finite graphs since they can be used to prove logarithmic Sobolev and subgaussian concentration inequalities \cite{O-V-2000, C-H-L-T}. In this section, we establish the Talagrand-type inequalities on discrete infinite graphs and compare our Hilbert manifold distance with the classical Wasserstein distance. 
	\subsection{Talagrand inequalities}
	In this subsection, we establish some local Talagrand-type inequalities. First, we recall that in \eqref{B-2-10} and \eqref{B-2-11} obtained that for $\sigma\in\mathcal{T}_{\rho}\mathcal{M}$,
	\begin{align}\label{D-1-1}
		\frac{\left(\inf\limits_{i\in\bbn}\rho_i\right)^2	}{\sup\limits_{i\in\bbn}\rho_i} g_{\rho}^{\Phi}(\sigma, \sigma)\le	\|\sigma\|_{l_{\bbm}^2}^2\le 2\left(\sup\limits_{i\in\bbn}\rho_i\right)	g_{\rho}^{\Phi}(\sigma, \sigma).
	\end{align}
\begin{theorem}\label{T4.1}
	Let \( \mathcal{G} = (\mathcal{V}, \mathcal{E},\bbm) \) be an infinite graph with vertex set \( \mathcal{V} = \{a_1, \dots, a_N, \dots\} \), edge set \( \mathcal{E} \), weight function $\bbm$, and a constant \( \beta > 0 \). For each $\mu=\left(\mu_i\right)_{i=1}^{\infty}\in {\rm Int}\mathcal{M}\cap l^{\infty}(\bbr)$ and a subset $\mathcal{C}_1$ of ${\rm Int}\mathcal{M}\cap l^{\infty}(\bbr)$ with 
	\begin{align}\label{D-1-3}
	0<\nu^{\mathcal{C}_1}_{\inf}:=\inf\limits_{\nu\in\mathcal{C}_1}\inf\limits_{i\in\bbn}\nu_i\le\sup\limits_{\nu\in\mathcal{C}_1}\sup\limits_{i\in\bbn}\nu_i=:\nu^{\mathcal{C}_1}_{\sup}<\infty,
	\end{align}
	 there exists a potential function $\Phi=\left(\Phi_i\right)_{i=1}^{\infty}$ $\left(\Phi_i=-\log \mu_i\right)$ on $\mathcal{V}$ and a constant $\kappa=\kappa(\nu^{\mathcal{C}_1}_{\inf}, \nu^{\mathcal{C}_1}_{\sup}, \mu, \bbm)>0$ such that for any $\nu=\left(\nu_i\right)_{i=1}^{\infty}\in\mathcal{C}_1$, we have the following local Talagrand-type inequality
	\begin{align}\label{D-1-2}
		d_{\Phi}^2(\mu,\nu)\le \kappa H(\nu|\mu),
	\end{align}
	where $H(\nu|\mu)=\sum\limits_{i=1}^{\infty}m_i\nu_i\log\frac{\nu_i}{\mu_i}$.
\end{theorem}
\begin{proof}
	With loss of generality, we assume that $\beta=1$ and fix $\delta=\frac{1}{2}$ in Remark \ref{R2.1}. Let $\Phi_i=-\log \mu_i,~i\in\bbn$, we have $H(\nu|\mu)\ge0$, since $\mu$ is the Gibbs distribution for the relative free energy functional $\mathcal{F}(
\nu|\bbm)$. Moreover, we note that
\[\mathcal{F}(
\nu|\bbm)=H(\nu|\mu).\] 
Furthermore, we can combine Proposition \ref{P2.2}, Remark \ref{R2.1}, \eqref{D-1-3}, and \[0<\inf\limits_{i\in\bbn}\mu_i\le\sup\limits_{i\in\bbn}\mu_i<\infty\]
to see there exist some uniform $N_0$ such that \eqref{C-1-9} hold for solutions of \eqref{FP} with any initial data $\nu\in\mathcal{C}_1$.
	We define 
	\begin{align*}
		\mathcal{C}_2:=\left\{\nu\in \mathcal{M}~~\Big|~~0<C_3\le\inf\limits_{i\in\bbn} \rho_i\le \sup\limits_{i\in\bbn} \rho_i\le C_4<\infty\right\},
	\end{align*}
	where
	\begin{align}\label{D-1-4}
		C_3:=\min\left\{\frac{1}{2}\nu^{\mathcal{C}_1}_{\inf},~~\frac{e^{-\frac{4\|\Phi\|_{l^\infty}}{M_{N_0}}}}{4},~~ \frac{ M_{N_0}}{4(2\|\Phi\|_{l^\infty}+1)}\right\},
	\end{align}
	
	\begin{align}\label{D-1-5}
		C_4:=\max\left\{2\nu^{\mathcal{C}_1}_{\sup},\quad \frac{e^{\frac{2\|\Phi\|_{l^\infty}}{ C_3M_{N_0}}}}{2M_{N_0}},~~\frac{2\|\Phi\|_{l^\infty}+2}{ M_{N_0}},~\frac{2}{M_{N_0}}\right\}.
	\end{align}
Clearly, we have $\mathcal{C}_1\subset\mathcal{C}_2$ by the definition of $C_3$ and $C_4$ in \eqref{D-1-4}-\eqref{D-1-5}.   With the above preparations, we are going to show that there is a constant $\kappa(\nu^{\mathcal{C}_1}_{\inf}, \nu^{\mathcal{C}_1}_{\sup}, \mu, \bbm)$ such that 
	\[	d_{\Phi}^2(\mu,\nu)\le \kappa H(\nu|\mu)\]
	for any $\nu\in\mathcal{C}_1$. Before moving to the proof details, we summarize our proof strategy below.\newline 
    
	\begin{itemize}
		\item First, we let $\rho(t)$ be the solution of FPE \eqref{FP} for $\beta=1$ with initial value $\nu$. Then, we use (3) of Theorem \ref{T2.1} to see $\rho(t)\in \mathcal{C}_2$ for all $t\ge0$.
		\item Second, we use the fact that $\mu$ is the Gibbs distribution with Remark \ref{R2.2} and \eqref{NewC_3} to find 
		\[\|\rho(t)-\mu\|_{l_{\rm \bbm}^2}^2\le\frac{\sup_i \mu_i}{\inf_i \mu_i}\|\nu-\mu\|_{l_{\rm \bbm}^2}^2e^{-\frac{C_3}{2C_4}\frac{\left(\inf_i \mu_i\right)}{\left(\sup_i \mu_i\right)}t}=:\|\nu-\mu\|_{l_{\rm \bbm}^2}^2C_5e^{-C_6t}.\]
		This implies for $T=\frac{\log (4C_5)}{C_6}>0$ ($C_5=\frac{\sup_i \mu_i}{\inf_i \mu_i}\ge1$), 
		\[\|\rho(T)-\mu\|_{l_{\rm \bbm}^2}^2\le \frac{1}{4}\|\nu-\mu\|_{l_{\rm \bbm}^2}^2\le \frac{1}{2}\left(\|\rho(T)-\mu\|_{l_{\rm \bbm}^2}^2+\|\rho(T)-\nu\|_{l_{\rm \bbm}^2}^2\right).\]
		Therefore, we get $T$ to guarantee
		\begin{align}\label{D-1-6}
			\|\rho(T)-\nu\|_{l_{\rm \bbm}^2}^2\ge \|\rho(T)-\mu\|_{l_{\rm \bbm}^2}^2.
			\end{align}
		\item Third, we use 
		\begin{align}\label{D-1-7}
				\frac{\di \mathcal{F}(\rho|\bbm)}{\di t}=-	g_{\rho}^{\Phi}\left(\frac{\di \rho}{\di t}, \frac{\di \rho}{\di t}\right)
				\end{align} to find 
			\[ \mathcal{F}(\nu|\bbm)- \mathcal{F}(\rho(T)|\bbm)\ge \frac{1}{T}d_{\Phi}^2(\nu,\rho(T)).\]
			\item Finally, we obtain 
			\[d_{\Phi}^2(\nu,\rho(T))\le T \mathcal{F}(\nu|\bbm) \quad \text{and}\quad d_{\Phi}^2(\rho(T),\nu)\le 2T\left(\frac{C_4}{C_3}\right)^{2} \mathcal{F}(
\nu|\bbm). \]
			These yield
			\[d_{\Phi}^2(\nu,\mu)\le 2T\left(1+2\left(\frac{C_4}{C_3}\right)^{2}\right)\mathcal{F}(
\nu|\bbm)=:\kappa \mathcal{F}(
\nu|\bbm)=\kappa H(\nu|\mu).\]
	\end{itemize}
	Since the first two steps are easy to check, we start from the third step. Integrating \eqref{D-1-7} from $0$ to $T$, we use \eqref{D-1-1} and $\rho(t)\in\mathcal{C}_2$ to get 
	\begin{align}\label{D-1-8}
		\begin{aligned}
			 \mathcal{F}(
\nu|\bbm)- \mathcal{F}(
\rho(T)|\bbm)&=\int\limits_{0}^{T} g_{\rho(t)}^{\Phi}\left(\frac{\di \rho(t)}{\di t}, \frac{\di \rho(t)}{\di t}\right)\di t\ge \frac{1}{T}\left(\int\limits_{0}^{T} \sqrt{g_{\rho(t)}^{\Phi}\left(\frac{\di \rho(t)}{\di t}, \frac{\di \rho(t)}{\di t}\right)}\di t\right)^2\\
			&\ge\frac{1}{T}\left(\int\limits_{0}^{T} \frac{1}{\sqrt{2C_4}}\|\dot{\rho}\|_{l_{\bbm}^2}\di t\right)^2\ge \frac{1}{2TC_4}\|\rho(T)-\nu\|_{l_{\rm \bbm}^2}^2.
		\end{aligned}
		\end{align}
		Moreover, we use \eqref{B-2-12} to find 
			\begin{align*}
			\begin{aligned}
				 \mathcal{F}(
\nu|\bbm)- \mathcal{F}(
\rho(T)|\bbm)&=\int\limits_{0}^{T} g_{\rho(t)}^{\Phi}\left(\frac{\di \rho(t)}{\di t}, \frac{\di \rho(t)}{\di t}\right)\di t\ge \frac{1}{T}\left(\int\limits_{0}^{T} \sqrt{g_{\rho(t)}^{\Phi}\left(\frac{\di \rho(t)}{\di t}, \frac{\di \rho(t)}{\di t}\right)}\di t\right)^2\\
				&\ge\frac{1}{T}d_{\Phi}^2(\nu,\rho(T)).
			\end{aligned}
		\end{align*}
		This implies \begin{align}\label{D-1-9}
			d_{\Phi}^2(\nu,\rho(T))\le T( \mathcal{F}(
\nu|\bbm)- \mathcal{F}(
\rho(T)|\bbm))\le T \mathcal{F}(
\nu|\bbm).
			\end{align}
			Next, we combine the convexity of $\mathcal{C}_2$, $\rho(T)\in\mathcal{C}_2$ and $ \mu\in \mathcal{C}_2$ to derive curve $\gamma(t)=t\rho(T)+(1-t)\mu \in  \mathcal{C}_2$ for $t\in[0,1]$. Note that curve $\gamma(t)$ connect $\mu$ and $\rho(T)$. Then, we can use left hand of \eqref{D-1-1} to see 
			\begin{align*}
						d_{\Phi}^2(\rho(T),\mu)&\le \left(\int\limits_{0}^{1} \sqrt{g_{\gamma(t)}^{\Phi}\left(\frac{\di \gamma(t)}{\di t}, \frac{\di \gamma(t)}{\di t}\right)}\di t\right)^2\\
						&= \left(\int\limits_{0}^{1} \sqrt{g_{\gamma(t)}^{\Phi}\left(\rho(T)-\mu,\rho(T)-\mu\right)}\di t\right)^2\\
						&\le\left(\int\limits_{0}^{1} \sqrt{\frac{C_4}{(C_3)^2}\|\rho(T)-\mu\|_{l_{\bbm}^2}^2}\di t\right)^2=\frac{C_4}{(C_3)^2}\|\rho(T)-\mu\|_{l_{\bbm}^2}^2.
			\end{align*}
			We combine \eqref{D-1-6} and \eqref{D-1-8} to find
			\begin{align}\label{D-1-10}
				\begin{aligned}
					d_{\Phi}^2(\rho(T),\mu)&\le \frac{C_4}{(C_3)^2}\|\rho(T)-\mu\|_{l_{\bbm}^2}^2\le \frac{C_4}{(C_3)^2}\|\rho(T)-\nu\|_{l_{\bbm}^2}^2\\
					&\le 2T\left(\frac{C_4}{C_3}\right)^2\left(	 \mathcal{F}(
\nu|\bbm)- \mathcal{F}(
\rho(T)|\bbm)\right)\le 2T\left(\frac{C_4}{C_3}\right)^2 \mathcal{F}(
\nu|\bbm).
								\end{aligned}
			\end{align}
			Finally, we combine \eqref{D-1-9} with \eqref{D-1-10} to obtain 
			\begin{align*}
				\begin{aligned}
				d_{\Phi}^2(\nu,\mu)\le2 d_{\Phi}^2(\rho(T),\nu)+2d_{\Phi}^2(\rho(T),\mu)\le 2T\left(1+2\left(\frac{C_4}{C_3}\right)^{2}\right)\mathcal{F}(
\nu|\bbm)=\kappa H(\nu|\mu).
							\end{aligned}
						\end{align*}
						Note that $\kappa$ only depends on $\nu^{\mathcal{C}_1}_{\inf}, \nu^{\mathcal{C}_1}_{\sup}, \mu, $ and $\bbm$.
\end{proof}
		\subsection{Compared with Wasserstein distance}
		In this subsection, we compare our distance $d_{\Phi}(\mu,\nu)$ with the $1$-Wasserstein distance $W_1(\mu,\nu)$. \newline
		
		By Kantorovich-Rubinstein theorem \cite{Villani2009,Villani20003}, we have 
		\begin{align}\label{D-1-11}
		W_1(\mu,\nu)=\sup\limits_{\psi: {\rm Lip}(\psi)\le1\atop\|\psi\|_{l^\infty}\le1}\left|\sum\limits_{i=1}^{\infty}m_i\psi_i(\mu_i-\nu_i)\right|.
		\end{align}
		Now, we want to show that 
		\[W_1(\mu,\nu)\le 2	d_{\Phi}(\mu,\nu).\]
		Let $\gamma(t):[0,1]\to \mathcal{M}$ be a continuously differentiable curve with $\gamma(0)=\nu, ~\gamma(1)=\mu.$ For any test function $\psi$ with ${\rm Lip} \psi\le1$ and $\|\psi\|_{l^\infty}\le1$, we choose $p_t$ to be the identification \eqref{bijection} for $\dot{\gamma}(t)$ and then we have
	\[	\left|\sum\limits_{i=1}^{\infty}m_i\psi_i(\mu_i-\nu_i)\right|=\left|\int_0^1\sum\limits_{i=1}^{\infty}m_i\psi_i \dot{\gamma}_i(t)\di t\right|=\left|\int_0^1\sum\limits_{i=1}^{\infty}\sum_{j=1}^{\infty}m_im_j\psi_i \tau_{ij}^{\Phi}(\gamma(t))(p_i-p_j)\di t\right|.\]
	Furthermore, we combine the Cauchy inequality, $\|\psi\|_{l^\infty}\le1$, \eqref{B-2-6}, and \eqref{B-2-13} to derive 
	\begin{align*}
		&\quad		\left|\int_0^1\sum\limits_{i=1}^{\infty}\sum_{j=1}^{\infty}m_im_j\psi_i \tau_{ij}^{\Phi}(\gamma(t))(p_i-p_j)\di t\right|\\
		&\le \int_0^1\left(\sum\limits_{i=1}^{\infty}\sum_{j=1}^{\infty}m_im_j\tau_{ij}^{\Phi}(\gamma(t))(\psi_i)^2\right)^{\frac{1}{2}} \left(\sum\limits_{i=1}^{\infty}\sum_{j=1}^{\infty}m_im_j\tau_{ij}^{\Phi}(\gamma(t))(p_i-p_j)^2\right)^{\frac{1}{2}}\di t\\
		&\le\int_0^1\left(\sum\limits_{i=1}^{\infty}\sum_{j=1}^{\infty}m_im_j\tau_{ij}^{\Phi}(\gamma(t))\right)^{\frac{1}{2}} \left(\sum\limits_{i=1}^{\infty}\sum_{j=1}^{\infty}m_im_j\tau_{ij}^{\Phi}(\gamma(t))(p_i-p_j)^2\right)^{\frac{1}{2}}\di t\\
		&\le\sqrt{2} \int_0^1 \left(\sum\limits_{i=1}^{\infty}\sum_{j=1}^{\infty}m_im_j\tau_{ij}^{\Phi}(\gamma(t))(p_i-p_j)^2\right)^{\frac{1}{2}}\di t\\
        &=2\int_0^1 \sqrt{g_{\gamma(t)}^\Phi(\dot{\gamma}(t), \dot{\gamma}(t))} \, \di t=2\mathcal{L}(\gamma).
	\end{align*}
	Here, we used\[\sum\limits_{i=1}^{\infty}\sum_{j=1}^{\infty}m_im_j\tau_{ij}^{\Phi}(\gamma(t))\le \sum\limits_{i=1}^{\infty}\sum_{j=1}^{\infty}\max\left\{m_im_j\gamma_j(t),m_jm_i\gamma_i(t)\right\}\le 2\sum\limits_{i=1}^{\infty}m_i\gamma_i(t)=2.\]
Finally, the arbitrariness of $\psi$ and $\gamma$ implies
	\[W_1(\mu,\nu)\le 2d_{\Phi}(\mu,\nu)\,,\]
leading to the desired estimate.

\bigskip

%
%
%
%
%
%
\bibliographystyle{amsplain}
\appendix
\section{The map $\tau_{\Phi(\rho)}$ is an injection}\label{App-A}
	\setcounter{equation}{0}
We prove that the map \(\tau_{\Phi(\rho)}:\mathcal W\to\bT_\rho\mathcal M\) is injective under the standing assumption
\[
0<\rho_{\inf}:=\inf_{i\in\bbn}\rho_i\le\sup_{i\in\bbn}\rho_i=:\rho_{\sup}<\infty.
\]
By \eqref{B-2-2}, this implies
\[
0<\rho_{\inf}\le\tau_{ij}^{\Phi}(\rho)\le\rho_{\sup}<\infty,\qquad i,j\in\bbn.
\]
Suppose that \(\tau_{\Phi(\rho)}([p])=0\). Then, for every \(i\in\bbn\),
\[
\sum_{j=1}^{\infty}m_j\tau_{ij}^{\Phi}(\rho)(p_i-p_j)=0.
\]
Equivalently,
\begin{equation}\label{A.1}
	\frac{\sum_{j=1}^{\infty}m_j\tau_{ij}^{\Phi}(\rho)(p_j-p_i)}
	{\sum_{j=1}^{\infty}m_j\tau_{ij}^{\Phi}(\rho)}
	=0,\qquad i\in\bbn.
\end{equation}
We will show that \(p_i=p_j\) for all \(i,j\in\bbn\). Since \(\mathcal W\) is the quotient space modulo constants, this proves injectivity. We split the proof into four cases.

\medskip
\noindent
\textbf{Case 1: the supremum is attained.}
Assume that there exists \(i_\infty\in\bbn\) such that
\[
p_{i_\infty}=\sup_{i\in\bbn}p_i.
\]
If \(p\) is not constant, then there exists \(j_0\in\bbn\) such that
\[
p_{j_0}<p_{i_\infty}.
\]
Since \(m_{j_0}>0\) and \(\tau_{i_\infty j_0}^{\Phi}(\rho)>0\), we have
\[
\sum_{j=1}^{\infty}m_j\tau_{i_\infty j}^{\Phi}(\rho)(p_j-p_{i_\infty})<0,
\]
because each term is non-positive and the \(j_0\)-term is strictly negative. Hence
\[
\frac{\sum_{j=1}^{\infty}m_j\tau_{i_\infty j}^{\Phi}(\rho)(p_j-p_{i_\infty})}
{\sum_{j=1}^{\infty}m_j\tau_{i_\infty j}^{\Phi}(\rho)}
<0,
\]
which contradicts \eqref{A.1}. Therefore, if the supremum is attained, \(p\) must be constant.

\medskip
\noindent
\textbf{Case 2: \(p\) is bounded but the supremum is not attained.}
Assume
\[
-\infty<p_{\inf}:=\inf_{i\in\bbn}p_i<p_{\sup}:=\sup_{i\in\bbn}p_i<\infty.
\]
We prove that this is impossible. Define
\[
I_\infty:=\left\{i\in\bbn:\frac{p_i-p_{\inf}}{p_{\sup}-p_{\inf}}=1\right\}
\]
and, for \(n\in\bbn\),
\[
I_n:=\left\{i\in\bbn:1-\frac1n\le\frac{p_i-p_{\inf}}{p_{\sup}-p_{\inf}}<1-\frac1{n+1}\right\}.
\]
Since \(p_{\inf}<p_{\sup}\), at least one of these sets is non-empty. Let \(n_0\) be the smallest integer such that \(I_{n_0}\neq\emptyset\).

If \(I_\infty\neq\emptyset\), then choosing \(i\in I_\infty\) and \(j\in I_{n_0}\), we have \(p_j<p_i\). As in Case 1, this gives
\[
\frac{\sum_{j=1}^{\infty}m_j\tau_{ij}^{\Phi}(\rho)(p_j-p_i)}
{\sum_{j=1}^{\infty}m_j\tau_{ij}^{\Phi}(\rho)}
<0,
\]
which contradicts \eqref{A.1}. Hence \(I_\infty=\emptyset\).

Since \(I_\infty=\emptyset\) and \(p_{\sup}\) is the supremum, there exists an increasing sequence \(\{n_l\}_{l\ge0}\) such that \(I_{n_l}\neq\emptyset\) and \(n_l\to\infty\). Fix \(n_2>n_0+1\) with \(I_{n_2}\neq\emptyset\). For \(i\in I_{n_2}\) and \(j\in I_{n_0}\), we have
\[
p_j-p_i\le-\frac{n_2-1-n_0}{n_2(n_0+1)}(p_{\sup}-p_{\inf}).
\]
Therefore,
\[
\frac{\sum_{j\in I_{n_0}}m_j\tau_{ij}^{\Phi}(\rho)(p_j-p_i)}
{\sum_{j=1}^{\infty}m_j\tau_{ij}^{\Phi}(\rho)}
\le-\frac{\rho_{\inf}}{\rho_{\sup}}\left(\sum_{j\in I_{n_0}}m_j\right)
\frac{n_2-1-n_0}{n_2(n_0+1)}(p_{\sup}-p_{\inf})<0.
\]
Moreover, for \(i\in I_n\), we have
\[
p_j-p_i\le p_{\sup}-p_i\le\frac1n(p_{\sup}-p_{\inf}),\qquad j\in\bbn.
\]
Thus, for \(i\in I_{n_N}\) with \(n_N\) sufficiently large,
\begin{align*}
	&\frac{\sum_{j=1}^{\infty}m_j\tau_{ij}^{\Phi}(\rho)(p_j-p_i)}
	{\sum_{j=1}^{\infty}m_j\tau_{ij}^{\Phi}(\rho)}\\
	&\le-\frac{\rho_{\inf}}{\rho_{\sup}}\left(\sum_{j\in I_{n_0}}m_j\right)
	\frac{n_N-1-n_0}{n_N(n_0+1)}(p_{\sup}-p_{\inf})
	+\frac1{n_N}(p_{\sup}-p_{\inf})<0.
\end{align*}
This contradicts \eqref{A.1}. Hence the bounded non-constant case cannot occur.

\medskip
\noindent
\textbf{Case 3: the supremum is infinite.}
Assume \(p_{\sup}=+\infty\). Since \(p\in l_{\bbm}^2\), we have
\[
\sum_{j=1}^{\infty}m_j|p_j|\le\|p\|_{l_{\bbm}^2}.
\]
For every \(i\in\bbn\),
\begin{align*}
	\frac{\sum_{j=1}^{\infty}m_j\tau_{ij}^{\Phi}(\rho)(p_j-p_i)}
	{\sum_{j=1}^{\infty}m_j\tau_{ij}^{\Phi}(\rho)}
	&=\frac{\sum_{j=1}^{\infty}m_j\tau_{ij}^{\Phi}(\rho)p_j}
	{\sum_{j=1}^{\infty}m_j\tau_{ij}^{\Phi}(\rho)}-p_i\\
	&\le\frac{\rho_{\sup}}{\rho_{\inf}}\|p\|_{l_{\bbm}^2}-p_i.
\end{align*}
Since \(p_{\sup}=+\infty\), we may choose \(i_0\in\bbn\) such that
\[
p_{i_0}>\frac{\rho_{\sup}}{\rho_{\inf}}\|p\|_{l_{\bbm}^2}.
\]
Then
\[
\frac{\sum_{j=1}^{\infty}m_j\tau_{i_0j}^{\Phi}(\rho)(p_j-p_{i_0})}
{\sum_{j=1}^{\infty}m_j\tau_{i_0j}^{\Phi}(\rho)}
<0,
\]
contradicting \eqref{A.1}.

\medskip
\noindent
\textbf{Case 4: the infimum is negative infinity.}
Assume \(p_{\inf}=-\infty\). Similarly,
\begin{align*}
	\frac{\sum_{j=1}^{\infty}m_j\tau_{ij}^{\Phi}(\rho)(p_j-p_i)}
	{\sum_{j=1}^{\infty}m_j\tau_{ij}^{\Phi}(\rho)}
	&=\frac{\sum_{j=1}^{\infty}m_j\tau_{ij}^{\Phi}(\rho)p_j}
	{\sum_{j=1}^{\infty}m_j\tau_{ij}^{\Phi}(\rho)}-p_i\\
	&\ge-\frac{\rho_{\sup}}{\rho_{\inf}}\|p\|_{l_{\bbm}^2}-p_i.
\end{align*}
Since \(p_{\inf}=-\infty\), we may choose \(i_0\in\bbn\) such that
\[
p_{i_0}<-\frac{\rho_{\sup}}{\rho_{\inf}}\|p\|_{l_{\bbm}^2}.
\]
Then
\[
\frac{\sum_{j=1}^{\infty}m_j\tau_{i_0j}^{\Phi}(\rho)(p_j-p_{i_0})}
{\sum_{j=1}^{\infty}m_j\tau_{i_0j}^{\Phi}(\rho)}
>0,
\]
again contradicting \eqref{A.1}.

The four cases exhaust all possible alternatives. In each non-constant case, we obtain a contradiction to \eqref{A.1}. Hence \(p\) must be constant. Therefore \([p]=0\) in \(\mathcal W\), which proves that \(\tau_{\Phi(\rho)}\) is injective.

\section*{Acknowledgments} 
The authors are deeply indebted and grateful to the anonymous reviewers whose extensive comments and suggestions have substantially improved the quality of this paper.
\section*{Conflict of interest statement}
The authors declare no conflicts of interest.

\section*{Data availability statement}
The data supporting the findings of this study are available from the corresponding author upon reasonable request.

\section*{Ethical statement}
The authors declare that this manuscript is original, has not been published before, and is not currently being considered for publication elsewhere. The study was conducted by the principles of academic integrity and ethical research practices. All sources and contributions from others have been properly acknowledged and cited. The authors confirm that there is no fabrication, falsification, plagiarism, or inappropriate manipulation of data in the manuscript.

\end{document}